\documentclass[a4paper,11pt,reqno]{amsart}

\usepackage[utf8]{inputenc}
\usepackage[T1]{fontenc}
\usepackage{lmodern}
\usepackage[english]{babel}
\usepackage{microtype}

\usepackage{amsmath,amssymb,amsfonts,amsthm}
\usepackage{mathtools,accents}
\usepackage{mathrsfs}
\usepackage{aliascnt}
\usepackage{braket}
\usepackage{bm}

\usepackage[a4paper,margin=3cm]{geometry}
\usepackage[citecolor=blue,colorlinks]{hyperref}

\usepackage{enumerate}
\usepackage{xcolor}

\makeatletter
\g@addto@macro\@floatboxreset\centering
\makeatother

\makeatletter
\def\newaliasedtheorem#1[#2]#3{
	\newaliascnt{#1@alt}{#2}
	\newtheorem{#1}[#1@alt]{#3}
	\expandafter\newcommand\csname #1@altname\endcsname{#3}
}
\makeatother

\numberwithin{equation}{section}

\newtheoremstyle{slanted}{\topsep}{\topsep}{\slshape}{}{\bfseries}{.}{.5em}{}

\theoremstyle{plain}
\newtheorem{theorem}{Theorem}[section]
\newaliasedtheorem{proposition}[theorem]{Proposition}
\newaliasedtheorem{lemma}[theorem]{Lemma}
\newaliasedtheorem{corollary}[theorem]{Corollary}
\newaliasedtheorem{counterexample}[theorem]{Counterexample}

\theoremstyle{definition}
\newaliasedtheorem{definition}[theorem]{Definition}
\newaliasedtheorem{question}[theorem]{Question}
\newaliasedtheorem{conjecture}[theorem]{Conjecture}

\theoremstyle{remark}
\newaliasedtheorem{remark}[theorem]{Remark}
\newaliasedtheorem{example}[theorem]{Example}

\newcommand{\setR}{\mathbb{R}}

\newcommand{\setH}{\mathbb{H}}
\newcommand{\Vv}{V}
\newcommand{\Gg}{\mathfrak{g}}
\newcommand{\Gh}{\mathfrak{h}}
\newcommand{\Vvh}{V^{\mathfrak{h}}}
\newcommand{\Gw}{W}

\let\altphi\phi
\let\phi\varphi
\let\varphi\altphi
\let\altphi\undefined

\newcommand{\p}{\oplus}
\newcommand\numleq[1]%
{\stackrel{\scriptscriptstyle(\mkern-1.5mu#1\mkern-1.5mu)}{=}}

\DeclareMathOperator{\de}{d}

\newfont{\tmpf}{cmsy10 scaled 2500}

\def\XXint#1#2#3{{\setbox0=\hbox{$#1{#2#3}{\int}$ }
		\vcenter{\hbox{$#2#3$ }}\kern-.6\wd0}}

\begin{document}
	
	\title{Pauls rectifiable and purely Pauls unrectifiable smooth hypersurfaces}
 
	\author[G. Antonelli and E. Le Donne]{Gioacchino Antonelli and Enrico Le Donne}
\address{\textsc{Gioacchino Antonelli}: 
Scuola Normale Superiore, Piazza dei Cavalieri, 7, 56126 Pisa, Italy}
\email{gioacchino.antonelli@sns.it}
\address{\textsc{Enrico Le Donne}: 
Dipartimento di Matematica, Universit\`a di Pisa, Largo B. Pontecorvo 5, 56127 Pisa, Italy \\
\& \\
University of Jyv\"askyl\"a, Department of Mathematics and Statistics, P.O. Box (MaD), FI-40014, Finland}
\email{enrico.ledonne@unipi.it}
	
	\renewcommand{\subjclassname}{%
 \textup{2010} Mathematics Subject Classification}
 \thanks{
 E.L.D. was partially supported by the Academy of Finland (grant
288501
`\emph{Geometry of subRiemannian groups}' and by grant
322898
`\emph{Sub-Riemannian Geometry via Metric-geometry and Lie-group Theory}')
and by the European Research Council
 (ERC Starting Grant 713998 GeoMeG `\emph{Geometry of Metric Groups}').
}
	
	\maketitle
	\begin{abstract}
	This paper is related to the problem of finding a good notion of rectifiability in sub-Riemannian geometry. In particular, we study which kind of results can be expected for smooth hypersurfaces in Carnot groups. Our main contribution will be a consequence of the following result: there exists a $C^{\infty}$ hypersurface $S$ without characteristic points that has uncountably many pairwise non-isomorphic tangent groups on every positive-measure subset. The example is found in a Carnot group of topological dimension 8, it has Hausdorff dimension 12 and so we use on it the Hausdorff measure $\mathcal{H}^{12}$. As a consequence, we show that for every Carnot group of Hausdorff dimension 12, any Lipschitz map defined on a subset of it with values in $S$ has $\mathcal{H}^{12}$-null image. In particular, we deduce that this smooth hypersurface cannot be Lipschitz parametrizable by countably many maps each defined on some subset of some Carnot group of Hausdorff dimension $12$. As main consequence we have that a notion of rectifiability proposed by S.Pauls is not equivalent to one proposed by B.Franchi, R.Serapioni and F.Serra Cassano, at least for arbitrary Carnot groups. In addition, we show that, given a subset $U$ of a homogeneous subgroup of Hausdorff dimension $12$ of a Carnot group, every bi-Lipschitz map $f:U\to S$ satisfies $\mathcal{H}^{12}(f(U))=0$. Finally, we prove that such an example does not exist in Heisenberg groups: we prove that all $C^{\infty}$-hypersurfaces in $\mathbb H^n$ with $n\geq 2$ are countably $\mathbb{H}^{n-1}\times\mathbb R$-rectifiabile according to Pauls' definition, even with bi-Lipschitz maps.
	\end{abstract}

 	\tableofcontents
	
	\section{Introduction}
	\textbf{State of the art:} Measure-theoretic notions of rectifiability in Carnot groups has been deeply studied in the last 20 years. 
	At the end of the 90's it was understood by the work of Ambrosio and Kircheim \cite{AK99} that the classical notion of $k$-rectifiability in metric spaces was not the correct one in this setting. Recall that a set is $k$-rectifiable iff it is $\mathcal{H}^k$-a.e. covered by Lipschitz images of subsets of $\mathbb R^k$. Indeed, in \cite{AK99} (see also more general results in \cite{Mag04}) they showed that the first Heisenberg group $\mathbb H^1$ is purely $k$-unrectifiable for $k\geq 2$.  Then, in \cite{FSSC01}, in the setting of each Heisenberg group $\mathbb H^n$, Franchi, Serapioni and Serra Cassano, proposed a definition, named $C^1_{\rm H}$-hypersurface, that was meant to adapt the notion of regular surface from the Euclidean setting. Indeed, a $C^1_{\rm H}$-hypersurface is, locally around any point, the 0-level set of a $C^1_{\rm H}$-function with non-vanishing intrinsic gradient. In the same paper they proposed a definition of rectifiability for codimension-one sets in the setting of Heisenberg groups: a set in $\mathbb H^n$ is {\em codimension-one $C^1_{\rm H}$-rectifiable} iff it can be covered by countably many $C^1_{\rm H}$-hypersurfaces up to a $\mathcal{H}^{2n+1}$-null set. With this definition they proved that the reduced boundary of a set of finite perimeter is codimension-one $C^1_{\rm H}$-rectifiable thus showing that this could be a good notion of rectifiability at least for Heisenberg groups. Soon after they generalized these results to Carnot groups of step 2 in \cite{FSSC03a}.
	
	After the mentioned works, the class of $C^1_{\rm H}$-submanifolds has been intensely studied, also in general codimensions and beyond the Heisenberg setting. 
	For example, in \cite{FSSC03b} the authors
	provide an implicit function theorem for $C^1_{\rm H}$-functions in the setting of Carnot groups (see also \cite{CM06}), showing also that a $C^1_{\rm H}$-hypersurface is locally the boundary of a set of finite perimeter. In \cite{FSSC07} a systematic study of low dimensional and low codimensional $C^1_ {\rm H}$-submanifolds in $\mathbb H^n$ has been performed. In that reference the authors proposed also general definitions of {\em $k$-dimensional $C^1_{\rm H}$-rectifiability} in $\mathbb H^n$. 
	
	At the same time the problem of regularity of the parametrization of $C^1_{\rm H}$-submanifolds was widely studied: in \cite{FSSC06} the authors proposed the definitions of intrinsic Lipschitz function and intrinsic differentiable function in the setting of Heisenberg groups. In that reference they proved that a $C^1_{\rm H}$-submanifold is locally the graph of an intrinsic Lipschitz function - see also  \cite[Theorem A.5]{DMV19}. Thus, still in \cite{FSSC06}, the authors proposed, in the setting of Heisenberg groups and for any dimension $k$, a definition of rectifiability \emph{a priori} more general than $k$-dimensional $C^1_{\rm H}$-rectifiability, by using coverings with graphs of intrinsic Lipschitz functions. Later, in \cite{FSSC11}, they generalized the notion of intrinsic Lipschitz function and intrinsic differentiable function for arbitrary Carnot groups. They also showed, in $\mathbb{H}^n$, a Rademacher-type theorem for intrinsic Lipschitz functions defined on codimension-one subgroups and proved the equivalence between the two codimension-one definitions of rectifiability, i.e., the codimension-one $C^1_{\rm H}$-rectifiability and the one with covering by means of graphs of intrinsic Lipschitz functions. The Rademacher-type theorem and the equivalence of the two notions of codimension-one rectifiability has been extended to a larger class of Carnot groups in \cite{FMS14}, but not yet to all Carnot groups.
	
	A comprehensive presentation of the notion of intrinsic Lipschitz function is contained in \cite{FS16}. We point out that earlier studies of the notion of intrinsic differentiable function were also contained in \cite{ASCV06} in the setting of Heisenberg groups $\mathbb H^n$. In the reference it is showed that a $C^1_{\rm H}$-hypersurface is locally the graph of an uniformly intrinsic differentiable function that solves a \emph{Burger-type} equation. Generalizations of this result are contained in \cite{AS09}, in the setting of $C^1_{\rm H}$-submanifolds in $\mathbb H^n$, and in \cite{DiDonato18} for $C^1_{\rm H}$-hypersurfaces in the setting of Carnot groups of step 2. Further studies of metric properties of intrinsic Lipschitz graphs are contained also in \cite{CMPSC14}.\\
	
	At the beginning of 2000 Pauls proposed a different notion of rectifiability in \cite[Definition 4.1]{Pau04}. According to his definition, given $\mathbb G$ a Carnot group of Hausdorff dimension $Q$, a subset $E$ of another Carnot group is $\mathbb G$-rectifiable iff it can be covered $\mathcal{H}^Q$-a.e. by countably many Lipschitz images of subsets of $\mathbb G$.
	The relation between the two notions of differentiability - namely Pauls' one and the one(s) by Franchi, Serapioni and Serra Cassano - is far from being well understood. Notice that in \cite[Definition 3]{CP04} the authors propose another definition of rectifiability in which they allow $\mathbb G$ of the previous definition to be a homogeneous subgroup of a Carnot group.
	
	The query that has been left open is wether a $k$-dimensional $C^1_{\rm H}$-submanifold in a Carnot group is Lipschitz (or better bi-Lipschitz) parametrizable by subsets of  $k$-dimensional homogeneous subgroups of a Carnot group. One positive result in this direction has been obtained in \cite{CP04} in which the authors proved that any $C^1$-hypersurface in $\mathbb H^1$ is $N$-rectifiable, where $N$ is a vertical plane in $\mathbb H^1$ and the maps used for the parametrization are even defined on open sets. Then this result was improved by Bigolin and Vittone in \cite{BV10} showing that for any non-characteristic point of a $C^1$-hypersurface there exists a neighbourhood $U$ of it and bi-Lipschitz chart between an open subset of $N$ and $U$. In \cite{BV10} the authors also provided a partial negative answer to the query: they showed the existence of a $C^1_{\rm H}$-hypersurface in $\mathbb H^1$ and a point on it such that there is no bi-Lipschitz map from an open subset of $N$ and any neighbourhood of the point. 
	
	As far as we know, apart from these results, there are not general positive answers in the direction of parametrizing an arbitrary $C^1_{\rm H}$-hypersurface in a Carnot group - either with Lipschitz or bi-Lipschitz maps defined on measurable subsets of Carnot groups of the same dimension. Anyway recently, in a slightly different direction, Le Donne and Young in \cite{LDY19} proved that a sub-Riemannian manifold with constant Gromov-Hausdorff tangent $\mathbb G$, is countably $\mathbb G$-rectifiable, where $\mathbb G$ is a Carnot group. This result gives a possible way to show that smooth hypersurfaces in Carnot groups - sufficiently smooth in order to carry a sub-Riemannian structure - are $\mathbb G$-rectifiable for some $\mathbb G$, which is exactly what we do in the second part of this paper with smooth non-characteristic hypersurfaces in $\mathbb H^n$ with $n\geq 2$.
	We add that, using some ideas coming from the theory of quantitative differentiability, very recently Orponen \cite{Orp19} showed that any $C^{1,\alpha}_{\mathbb H}$-hypersurface with $\alpha>0$ in $\mathbb H^1$ is Lipschitz parametrizable with subsets of a vertical plane $N$. \\
	
	\textbf{Results:} In this paper we focus on the link between the two notions of differentiability explained above. The paper is essentially divided into two parts: in the first one we provide a \emph{negative}  result and in the second one we provide a \emph{positive} result. 
	
	In the first part we show the following (see \autoref{thm:T1.5}): 
	\begin{theorem}\label{thm:T1.5Intro}
		There exist a Carnot group $\mathbb G$ and a $C^{\infty}$ non-characteristic hypersurface $S\subseteq \mathbb G$ that is not Pauls Carnot rectifiable.
	\end{theorem}
	Pauls Carnot rectifiability is just a generalization of Pauls' rectifiability defined in \cite[Definition 4.1]{Pau04} in which we allow countably many Carnot groups, see \autoref{def:LipCarRect}. Our result shows that even very regular objects, such as smooth non-characteristic hypersurfaces, which for sure are rectifiable according to Franchi, Serapioni and Serra Cassano, are not so according to the definition given in \cite[Definition 4.1]{Pau04}.
	
	After, we show that such an example does not exist in the setting of $\mathbb H^n$ with $n\geq 2$ (see \autoref{thm:FinalTheorem} and \autoref{rem:BiLipPauls} for a more  exhaustive statement):
	\begin{theorem}\label{thm:FinalTheoremIntro}
		Let $S$ be a $C^{\infty}$-hypersurface in $\mathbb{H}^n$ with $n\geq 2$. Then $S$ is $\mathbb H^{n-1}\times \mathbb R$-rectifiable according to Pauls' definition of rectifiability {\em \cite[Definition 3]{CP04}}, even with bi-Lipschitz maps. \\
	\end{theorem}

	\textbf{Comments and ideas of the proofs:} To prove \autoref{thm:T1.5Intro}, whose proof is in Section \ref{sec:MainResults}, we will show the existence of a $C^{\infty}$-hypersurface $S$ - of Hausdorff dimension 12 in a Carnot group of topological dimension 8 - that cannot be $\mathcal{H}^{12}$-a.e.~covered by countably many Lipschitz images of subsets of Carnot groups of Hausdorff dimension 12. Notice that we also allow the Carnot groups to vary, thus using a more general definition of rectifiability with respect to the one given in \cite[Definition 4.1]{Pau04}. 
	
	We will actually show a more general property for $S$: for every Carnot group $\mathbb G$ of Hausdorff dimension 12, every Lipschitz map $f:U\subseteq \mathbb G\to S$  satisfies $\mathcal{H}^{12}(f(U))=0$ (see \autoref{cor:T1.5}). We will call this property \emph{purely Pauls Carnot unrectifiability} (\autoref{def:LipCarRect}), which implies that $S$ is not Pauls Carnot rectifiability, see \autoref{rem:Purely2}. The key property that $S$ enjoys is that every $\mathcal{H}^{12}$-positive subset of it has uncountably many points with pairwise non-isomorphic Carnot groups as tangents, see the statement and the proof of \autoref{thm:T1}, and \autoref{thm:T2}.
	
	The idea to build such a hypersurface is the following: at first, in \autoref{prop:R1}, we show the existence of a Carnot algebra of dimension $8$ that has uncountably many pairwise non-isomorphic Carnot subalgebras of dimension 7. This  is done by exploiting the existence of an uncountable family $\mathcal{F}$ of Carnot algebras of dimension 7 that are known to be pairwise non-isomorphic, see \cite{Gong98}. Notice that 7 is the minimum number for which this fact holds: there are, up to isomorphisms, only finitely many Carnot algebras of dimension $\leq 6$, see again \cite{Gong98}. Then we construct examples of smooth non-characteristic hypersurfaces $S$ in the Carnot group whose Lie algebra is the previous Carnot algebra of dimension 8, with the property that the tangent spaces to each $S$ form an uncountable subfamily of $\mathcal{F}$. With a particular choice of $S$ as in \autoref{rem:Ex2} we show that every $\mathcal{H}^{12}$-positive subset of $S$ has uncountably many points with pairwise non-isomorphic Carnot groups as tangents, as we planned.
	
	Having in our hands the pathological example $S$, we prove our main result, see \autoref{cor:T1.5}. We do it via a blow up analysis and using
    the area formula for Lipschitz maps between Carnot groups proved by Magnani in \cite{Mag01}.
	
	We point out that we also construct, in every Carnot group $\mathbb G$, a smooth non-characteristic hypersurface that has every subgroup of codimension-one of $\mathbb G$ as tangent, see \autoref{lem:EveryTangent} and \autoref{thm:T1}.
	
	\vspace{0.5cm}
	
	We also prove some strenghtening of \autoref{thm:T1.5Intro}. Namely, we show in \autoref{cor:T1} that our example $S$ is not, as we say, {\em bi-Lipschitz homogeneous rectifiable} (\autoref{def:TRect}). That is to say, it is impossible to $\mathcal{H}^{12}$-a.e. cover $S$ by countably many bi-Lipschitz images of subsets of metric spaces of Hausdorff dimension $12$ that have bi-Lipschitz equivalent tangents, a class larger than Carnot groups, which includes their subgroups. Actually, again, we prove more: we show that $S$ is \emph{purely bi-Lipschitz homogeneous unrectifiable} according to our definition \autoref{def:TRect}, after having provided a general criterion of \emph{purely bi-Lipschitz homogeneous unrectifiability} (\autoref{lem:KeyLemma}).
		
	Notice that, from this last result, it follows that $S$ is not rectifiable according to the countable bi-Lipschitz variant of the definition given in \cite[Definition 3]{CP04}, that is, the one that allows the parametrizing spaces to be homogeneous subgroups of Carnot groups, see also \autoref{rem:TangRectIsCarnotRect}. Nevertheless, we still are not able to prove that our counterexample is not rectifiable according to \cite[Definition 3]{CP04}, see \autoref{rem:NotAble}.
	
	We remark here that, from how we are going construct the example $S$, it follows that any tangent to $S$ is a Carnot group. Consequently, together with the previous discussed results, we immediately deduce that $S$ is also an example of metric space that cannot be Lipschitz parametrized by countably many of its tangents, see \autoref{rem:TangentRect}.

 	In \autoref{rem:SubRiemannian} we observe that $S$ has a structure of sub-Riemannian manifold and if we consider on it the sub-Riemannian distance it is still purely bi-Lipschitz homogeneous unrectifiable (\autoref{def:TRect}). \\
	
	To prove \autoref{thm:FinalTheoremIntro},  we will use \cite[Theorem 1.1.]{TY04}, \cite[Proposition 3.8.]{CMPSC14} and \cite[Theorem 1]{LDY19}. The proof is contained in Section \ref{sec:MainResults2}.
	
	The idea is the following: first we show that every non-characteristic hypersurface $S$ in $\mathbb H^n$ with $n\geq 2$ carries a structure of polarized manifold (\autoref{prop:EverywhereIsomorphic}). We show that the intersection of the horizontal bundle of $\mathbb H^n$ with the tangent bundle of $S$ is a step-2 bracket generating distribution (\autoref{prop:MainProp2}). This was already known from \cite[Theorem 1.1.]{TY04}, but we give a different proof based on simple explicit computations.
	
	Second we show that every sub-Riemannian structure on the polarized manifolds $S$ gives raise to a distance that is locally bi-Lipschitz equivalent to the distance on $S$ seen as subset of $\mathbb H^n$ (\autoref{prop:DistanceAreEquivalent}). We will call these latter distances the {\em intrinsic distance} and the {\em induced distance}, respectively.
	The equivalence is due to the general fact that in $\mathbb H^n$ with $n\geq 2$, the intrinsic distance and the induced distance on the graph of an intrinsic Lipschitz function are equivalent (\autoref{prop:EquivalentDistances}). We write the proof of this result, which was suggested to us by \"assler and Orponen, and which is actually a consequence of the result already known from \cite[Proposition 3.8.]{CMPSC14}. Third we use the fundamental tool \cite[Theorem 1]{LDY19} and the key fact that the tangents to the hypersurface are all isomorphic to $\mathbb H^{n-1}\times \mathbb R$ (\autoref{lem:IsoSub}). With these three steps we conclude the proof of \autoref{thm:FinalTheoremIntro}. \\
	
	\textbf{Structure of the paper:} The structure of the paper is the following: in Section \ref{sect:Preliminaries} we collect general definitions and tools that are useful for our aims. We recollect general definitions about Carnot groups and metric measure spaces; we recall the area formula for Lipschitz maps between Carnot groups; we revise some basic definitions and statements about $C^1_{\rm H}$-hypersurfaces, showing in particular that a $C^1_{\rm H}$-hypersurface has Hausdorff dimension $(Q-1)$ - being $Q$ the Hausdorff dimension of the group in which it lives - and $\mathcal{H}^{Q-1}$ is a locally doubling measure on it (\autoref{prop:C1HisQ-1}). We also stress that the tangent group to a $C^1_{\rm H}$-hypersurface is the Hausdorff tangent, which is a fact due to \cite[Theorem 3.1.1]{Koz15} (see \autoref{prop:GHTangent}).
	
	In Section \ref{sect:Rectifiability} we give different notions of rectifiability, such as bi-Lipschitz homogeneous rectifiability (\autoref{def:TRect}) and Pauls Carnot rectifiability (\autoref{def:LipCarRect}), the latter one being a variant of Pauls' rectifiability \cite[Definition 4.1]{Pau04}. Namely, a metric space of Hausdorff dimension $k$ is {\em Pauls Carnot rectifiable} if it is $\mathcal{H}^k$-a.e. covered by countably many Lipschitz images of subsets of Carnot groups of Hausdorff dimension $k$; a metric space of Hausdorff dimension $k$ is {\em bi-Lipschitz homogeneous rectifiable} if it is $\mathcal{H}^k$-a.e.~covered by countably many bi-Lipschitz images of subsets of metric spaces of Hausdorff dimension $k$ that have bi-Lipschitz equivalent tangents. We also give the notions of purely bi-Lipschitz homogeneous unrectifiability (\autoref{def:TRect}) and purely Pauls Carnot unrectifiability (\autoref{def:LipCarRect}), which are stronger version (see \autoref{rem:Purely} and \autoref{rem:Purely2}) of not being bi-Lipschitz homogeneous rectifiability and not being Pauls Carnot rectifiability, respectively. In \autoref{lem:KeyLemma} we provide a criterion for a metric space to be purely  bi-Lipschitz homogeneous unrectifiable. In \autoref{rem:TangRectIsCarnotRect} we discuss some definitions of rectifiability that are less general than \autoref{def:TRect}, while in \autoref{rem:LipCounterpart} we briefly discuss some Lipschitz counterpart of \autoref{def:TRect}.
	
	In Section \ref{sec:Construction} we construct a Carnot algebra of dimension $8$ that has uncountably many pairwise non-isomorphic Carnot subalgebras of dimension 7. The  construction is done in \autoref{prop:R1}. 
	
	In Section \ref{sec:MainResults} and Section \ref{sec:MainResults2} we show the main theorems we discussed above. Namely, first we show the main result \autoref{thm:T1.5Intro}, see \autoref{thm:T1.5}, and the streghted version we discussed above, see \autoref{cor:T1}. Second we show \autoref{thm:FinalTheoremIntro}, see \autoref{thm:FinalTheorem}.\\
	
	\textbf{Acknowledgments:} We would like to thank Katrin F\"assler and Tuomas Orponen for useful discussions around the topic of the paper. 

	\section{Preliminaries}\label{sect:Preliminaries}
	\subsection{Some standard definitions}\label{sub:Carnot}
	For definitions and theory about Carnot groups one can see \cite{LD17}, \cite{BLU07} and \cite[Section 2]{FSSC03a}. We  recall here some basic facts and terminology.
	 
	A Carnot group is a simply connected nilpotent Lie group whose Lie algebra is stratified and generated by the first stratum.
	If $\mathbb G$ is a Carnot group and $\Gg$ is its Lie algebra we thus have 
	$$
	\Gg=\Vv_1\p\dots\p\Vv_s,
	$$
	with $\Vv_{i+1}=[\Vv_1,\Vv_i]$ for every $1\leq i\leq s-1$, $\Vv_s\neq \{0\}$ and $[\Vv_1,\Vv_s]=\{0\}$. The number $s$ is called \emph{step} of the group $\mathbb G$. The dimension of the first stratum $\Vv_1$ is denoted by $m$ and the dimension of $\Gg$ by $n$.
	
	For a Carnot group the exponential map $\exp:\Gg\to\mathbb G$ is a diffeomorphism. Thus by means of this map, after a choice of a basis of $\Gg$, we can identify $\mathbb G$ with $\mathbb R^n$ with an operation $\cdot$ that can be explicitly written by making use of the Campbell-Hausdorff formula. We will use exponential coordinates
	$$
	(x_1,\dots,x_n)\mapsto \exp(x_1X_1+\dots+x_nX_n),
	$$
	where $\{X_1,\dots,X_n\}$ is a basis of $\Gg$ adapted to the stratification. Then $\{X_1,\dots,X_m\}$ is a basis of $\Vv_1$. With a little abuse of notation we will indicate with $X_i\in \Gg$ both a tangent vector at the identity element of $\mathbb G$ and the left-invariant vector field on $\mathbb G$ that agrees with it at the identity. We call $\{X_1,\dots,X_m\}$ a {\em basis of the horizontal space} $V_1$.
		
	On the Lie algebra $\Gg$ we have a family of linear maps $\delta_{\lambda}$ that act as 
	$$
	\delta_{\lambda}(v_i)=\lambda^i v_i, \qquad \mbox{if} \quad v_i\in \Vv_i.
	$$
	With an abuse of notations, we denote by $\delta_{\lambda}$ the group endomorphism on $\mathbb G$ with differential $\delta_{\lambda}$. Namely, by means of the exponential map we have  $\delta_{\lambda}:= \exp\circ \delta_{\lambda}\circ \exp^{-1}$ on $\mathbb G$ as well. 
	We call a homomorphism $\phi:\mathbb G\to\mathbb H$ between two Carnot groups a \emph{Carnot homomorphism} if 
	$$
	\phi\circ \delta_{\lambda} = \delta_{\lambda}\circ \phi, \qquad \forall \lambda>0.
	$$
	\begin{remark}\label{rem:CarnotHom}
		Every Carnot homomorphism induces a linear map $\phi_*:\Gg\to\Gh$, which is a Lie algebra homomorphism, such that $\phi_* \circ \delta_{\lambda} = \delta_{\lambda} \circ \phi_*$. From this property it easily follows that for every $1\leq i\leq s$ we get $\phi_*(\Vv_i)\subseteq \Vvh_i$, where $\Vv_i$ and $\Vvh_i$ are the $i$-th strata of $\Gg$ and $\Gh$, respectively.
	\end{remark}

	A {\em left-invariant homogeneous distance} $d:\mathbb G\times \mathbb G \to \mathbb R_{\geq 0}$ - sometimes we call it $d_\mathbb{G}$ - is a distance on $\mathbb G$ satisfying
	\begin{equation}\label{eqn:HomLeftInv}
	d(hg_1,hg_2)=d(g_1,g_2), \qquad d(\delta_{\lambda} g_1,\delta_{\lambda} g_2) = \lambda d(g_1,g_2), \qquad \forall h,g_1,g_2\in\mathbb G.
	\end{equation}   
	It follows from \cite[Proposition 3.5]{LD17} that such a distance is continuous with respect to the manifold topology on $\mathbb G$.
	
	There is a distinguished class of such distances, known as {\em Carnot-Carathéodory distances}. If we fix a norm $\|\cdot\|$ on the first stratum $\Vv_1$ of the Lie algebra $\Gg$ of $\mathbb G$, we can extend it left-invariantly on the horizontal bundle 
	$$
	\mathbb V_1(x):=(L_x)_{*}\Vv_1,
	$$
	for $x\in\mathbb G$, where $L_x$ is the left translation by $x$. We say that an absolutely continuous curve $\gamma:[0,1]\to \mathbb G$ is {\em horizontal} if
	$$
	\gamma'(t)\in \mathbb V_1(\gamma(t)), \qquad \mbox{for a.e.} \quad  t\in [0,1].
	$$
	We define
	$$
	d_{cc}^{\|\cdot\|}(x,y):= \inf\left\{\int_I \|\gamma'(t)\|: \quad \gamma(0)=x, \quad \gamma(1)=y, \quad \gamma \quad \mbox{horizontal}  \right\}.
	$$
	Chow-Rashevsky theorem states that this distance is finite. It is clearly homogeneous and left-invariant. 

	It is clear that any two homogeneous left-invariant distances $d_1$ and $d_2$ are equivalent: we write $d_1 \sim d_2$ and we mean that there exists $C\geq 1$ such that 
	$\frac{1}{C}d_1\leq d_2\leq Cd_1$.
	Then, from now on, we don't specify what homogeneous left-invariant distance we are choosing as we prove results that are true up to bi-Lipschitz maps.
	
	We remind that the Hausdorff dimension of a Carnot group $\mathbb G$ with respect to any left-invariant homogeneous distance is 
	$$
	Q:=\sum_{i=1}^s i\dim\Vv_i,
	$$ 	
	which we also call {\em homogeneous dimension}.
	
 	We recall now some definitions about metric spaces and metric measure spaces.
 	Given a metric space $(X,d)$, we indicate the open ball in the metric $d$ centered at $x$ of radius $r$ with $B^{d}(x,r)$. When the distance is clear we just write $B(x,r)$. The closed ball is indicated with $\overline{B}^{d}(x,r)$. We say that $(X,d)$ is proper if each closed ball $\overline{B}^{d}(x,r)$ is compact. We say that $(X,d_X)$ is {\em bi-Lipschitz equivalent} to $(Y,d_Y)$ if there exist a bijective map $f:X\to Y$ such that 
 	$$
 	\frac{1}{C}d_X(x_1,x_2)\leq d_Y(f(x_1),f(x_2))\leq Cd_X(x_1,x_2).
 	$$
 	
 	 We say that $(X,d,\mu)$ is a {\em metric measure space} if $(X,d)$ is a complete and separable metric space and $\mu$ is a Borel measure that is finite on bounded sets. A metric measure space $(X,d,\mu)$ is said to be locally doubling if for each $a\in X$ there exists $R_a>0$ and $C_a>0$ such that 
 	 $$
 	 \mu(B^{d}(x,2r))\leq C_a\mu(B^{d}(x,r))<+\infty \qquad \forall x\in B^{d}(a,R_a) \qquad  \forall 0\leq r\leq R_a. 
 	 $$
 	 For a proper locally doubling metric measure space as a consequence of Gromov compactness theorem we can say that for every $x\in X$, the set of Gromov-Hausdorff tangents $\mbox{Tan}(X,d,x)$ is nonempty. Indeed, we can say that for every sequence of positive numbers $\lambda_i\to 0$, up to subsequences 
 	 $$
 	 (X,\lambda_i^{-1}d,x) \to (X_{\infty},d_{\infty},x_{\infty})
 	 $$
 	 in the pointed Gromov-Hausdorff convergence. For general definitions and theory about (pointed) Gromov-Hausdorff convergence one can see \cite[Chapter 27]{Vil09} and \cite{BBI01}.
 	 
	 We indicate with $\mathcal{H}^k$ the Hausdorff measure of dimension $k$ associated to $d$ and with $\mathcal{S}^k$ the {\em spherical} Hausdorff measure of dimension $k$ associated to $d$. We denote with $\dim_H X$ the Hausdorff dimension of the metric space $X$. For these general definitions see \cite[2.10.2]{Fed69}. We indicate with $d_H$ the Hausdorff distance between sets, see \cite[2.10.21]{Fed69}.

	\subsection{Area formula for Lipschitz function between Carnot groups}\label{sub:Area}
	We shall recall the area formula for Lipschitz maps in Carnot groups which is due to Magnani. First we recall Rademacher theorem in this setting \cite[Theorem 3.9]{Mag01}.
	\begin{theorem}\label{thm:Rademacher}{\em [Magnani]}
		Let $\mathbb G$ and $\mathbb H$ be two Carnot groups. Let us call $Q$ the homogeneous dimension of $\mathbb G$. Then any Lipschitz map $f:A\subseteq\left(\mathbb G, d_{\mathbb{G}}\right)\to \left(\mathbb H,d_{\mathbb H}\right)$, where $A$ is a measurable set, is differentiable $\mathcal{H}^Q$- a.e., i.e., there exist, at $\mathcal{H}^Q$ a.e. point $x$ of $A$, a Carnot homomorphism $Df_x:\mathbb G\to \mathbb H$
		such that
		\begin{equation}\label{eqn:Diff}
		\lim_{y\in A, y\to x}\frac{d_{\mathbb H}(f(x)^{-1}f(y),Df_x(x^{-1}y))}{d_{\mathbb G}(x,y)}=0.
		\end{equation}
	\end{theorem}

	\begin{remark}\label{rem:DefinitionOfDifferential}
	For our aims we discuss here how $Df_x$ is defined. First of all, one takes $\{X_1,\dots,X_m\}$ a set of generators of the first stratum $\Vv_1$ of the Lie algebra $\Gg$ of $\mathbb G$. In \cite[Step 1 and Step 2]{Mag01} it is shown that if we take $z$ in the set 
	$$
	E:= \left\{\prod_{s=1}^{\gamma} \exp(a_{s}X_{i_s}): \gamma\in\mathbb N, \quad 1\leq i_s\leq m, \quad (a_{s})\in \mathbb Q^{\gamma}\right\}
	$$
	the limit 
	$$
	\lim_{x\delta_t z\in A, t\to 0} \delta_{1/t}\left(f(x)^{-1}f(x\delta_tz)\right)
	$$
	does exist in a subset of $A$ of $\mathcal{H}^Q$-full measure \cite[(3), Step 1, Proof of Theorem 3.9]{Mag01}. Now we get, being $E$ countable, that there exists a set of $\mathcal{H}^Q$-full measure in $A$, say $A_{\omega}$, such that for every $z \in E$ one has that the limit
	$$
	Df_x(z):= \lim_{x\delta_t z\in A, t\to 0} \delta_{1/t}\left(f(x)^{-1}f(x\delta_tz)\right)
	$$
	does exist for every $x\in A_{\omega}$. Then it can be shown \cite[Step 2]{Mag01} that on $A_{\omega}$, the map $Df_x$ can be extended to all $z\in \mathbb G$, exploiting the fact that $E$ is dense and just defining 
	$$
	Df_x(z)=\lim_{i\to+\infty} Df_x(z_i),
	$$
	where $z_i\in E$ and $z_i\to z$. In \cite[Step 2]{Mag01} it is shown that $Df_x$ is a Carnot homomorphism and in \cite[Step 3]{Mag01} it is shown that this differential satisfies \eqref{eqn:Diff}.
	\end{remark}

	\begin{definition}[Jacobian of a Lipschitz map]
		Given any Lipschitz $f:A\subseteq\left(\mathbb G, d_{\mathbb{G}}\right)\to \left(\mathbb H,d_{\mathbb H}\right)$ we can define the {\em Jacobian} 
		$$
		J_Q(Df_x):= \frac{\mathcal{H}^Q(Df_x(B(0,1)))}{\mathcal{H}^Q(B(0,1))},
		$$ 
		at any differentiability point $x$ of $f$.
	\end{definition}
	The following result is proved in \cite[Theorem 4.4]{Mag01}.
	\begin{theorem}\label{thm:AreaFormula}{\em [Magnani]}
		Given any Lipschitz map $f:A\subseteq\left(\mathbb G, d_{\mathbb{G}}\right)\to \left(\mathbb H,d_{\mathbb H}\right)$, where $A$ is a measurable set, we have
		$$
		\int_A J_Q(Df_x)\de\mathcal{H}^Q(x) = \int_\mathbb{H} \sharp(f^{-1}(y)\cap A)\de\mathcal{H}^Q(y).
		$$
	\end{theorem}
	
	\subsection{Parametrizations of a $C^1_{\rm H}$-hypersurface and its Hausdorff dimension}
	We give here some definitions about hypersurfaces in Carnot groups. One of our reference is the summary given in \cite[Section 2]{DiDonato18} and references therein, such as \cite{FSSC01, FMS14, FS16}.
	\begin{definition}[$C^1_{\rm H}$-hypersurfaces]\label{def:C1H}
	In a Carnot group $\mathbb G$, with  $\{X_1,\dots,X_m\}$ as basis of horizontal space, a subset $S$ is a {\em $C^1_{\rm H}$-hypersurface} if for all $p\in S$ there exists a neighbourhood $U$ of $p$ and a function $f:U\to \mathbb R$ with $X_1f,\dots,X_mf$ continuous in $U$, such that  
	\begin{equation}\label{rep}
	S\cap U = \{f=0\}\cap U,
	\end{equation}
	and $(X_1f,\dots,X_mf)$ does not vanish on $U$.
	\end{definition} 
	\begin{definition}[Characteristic points]\label{def:CharacteristicSurface}
		Let $\mathbb G$ be a Carnot group. Given $S$ a $C^1$-surface in $\mathbb G\equiv \mathbb R^n$, we say that $x\in S$ is a {\em characteristic point} for $S$ if 
		\begin{equation}\label{eqn:CharacteristicPoint}
		\mathbb V_1(x)\subseteq T_xS,
		\end{equation}
		where $\mathbb V_1(x)$ is the horizontal bundle at $x$ and $T_xS$ is the "Euclidean tangent" of $S$, i.e., the tangent space of $S$ seen as submanifold of $\mathbb G$. We shall use the term "Euclidean" in contrast of the intrinsic sub-Riemannian one. 
		
		We will say that a $C^1$-surface $S$ is {\em non-characteristic} if it does not have characteristic points as in \eqref{eqn:CharacteristicPoint}.
	\end{definition}
	\begin{remark}\label{rem:Characteristic}
	We identify $\mathbb G$ with $\mathbb R^n$ by means of exponential coordinates and we call $m$ the dimension of the first layer.
	If we take $f\in C^{1}(\mathbb G)$ 
	we will denote with $\nabla f|_x$ the {\em full gradient} of $f$ at $x$, i.e., the vector $\sum_{i=1}^{n}(\partial_{x_i}f)(x)\partial_{x_i}|_x$, 
	and with $\nabla_{\setH} f|_x$ the {\em horizontal gradient} of $f$ at $x$, i.e., the vector $\sum_{i=1}^{m}(X_if)(x)X_i|_x$.
	
	If $S$ is a $C^1$-hypersurface in $\mathbb G$, for every point $p\in S$ there exist an open neighbourhood $U_p$ of it and $f\in C^1(\mathbb G)$ such that
	\begin{equation}\label{eqn:Local}
	S\cap U_p= \{x \in U_p : f(x)=0\},
	\end{equation}
	with $\nabla f \neq 0$ on $S\cap U_p$. The Euclidean tangent space of $S$ at an arbitrary point $x\in S\cap U_p$ is
	\begin{equation}\label{DefOfTang}
	T_x S := \{v : \langle v, \nabla f|_x\rangle_x=0 \},
	\end{equation}
	where $\langle\cdot,\cdot\rangle_x$ is the usual inner product, i.e., $\langle \partial_{x_i}|_x, \partial_{x_j}|_x \rangle_x = \delta_{ij}$, and $v=\sum_{i=1}^{n} v_i\partial_{x_i}|_x$. Then $x\in U_p$ is a characteristic point \eqref{eqn:CharacteristicPoint} if and only if (see \eqref{DefOfTang}) it holds that $X_if(x)=0$ for all $i=1,\dots,m$.
	
	Thus a $C^1$-hypersurface $S$ in $\mathbb G$ with non-characteristic points is $C^1_{\rm H}$, because we have the representation in \eqref{eqn:Local} with $(X_1f,\dots,X_mf)\neq 0$ on $U_p$. 
	\end{remark}

    For $C^1_{\rm H}$-hypersurfaces we have a notion of tangent group \cite[page 14]{FSSC03a}. In the following definition we are identifying the group $\mathbb G$ with $\mathbb R^n$ by means of the exponential coordinates and we denote by $\{X_1,\dots,X_m\}$ a basis of the horizontal space.
    \begin{definition}[Tangent group to a $C^1_{\rm H}$-hypersurface]\label{Def:TangentGroup}
    	Given $S$ a $C^1_{\rm H}$-hypersurface, a point $p\in S$ and a representative $f$ around $p$ as in \eqref{rep}, we can define the {\em tangent group}, or the \emph{intrinsic tangent} of $S$ at $p$ as 
    	$$
    	T_p^{\rm I}S:= \left\{v\in\mathbb G\equiv \mathbb R^n: \sum_{i=1}^m v_i X_if(p)=0\right\}.
    	$$
    \end{definition}
    \begin{remark}
    	The tangent group defined in \autoref{Def:TangentGroup} is a subgroup of $\mathbb G$ and  does not depend on the representative $f$ given in \autoref{def:C1H} \cite[page 14]{FSSC03a}. Indeed, we have, see Section \ref{sec:HausdorffTangent}, that
    	$$
    	T_p^{\rm I}S=\lim_{\lambda\to 0}\delta_{\lambda^{-1}}(p^{-1}S).
    	$$
    \end{remark}
	We deal now with the problem of parametrization of a $C^1_{\rm H}$-hypersurface. We first introduce the notion of intrinsic Lipschitz function. See \cite[Section 2]{FMS14}.
	\begin{definition}\label{def:Cone}
		If $\mathbb G$ is a Carnot group and $\mathbb W$ and $\mathbb H$ are homogeneous subgroups such that $\mathbb W\cap \mathbb H=\{0\}$ and 
		$$
		\mathbb G= \mathbb W \cdot \mathbb H,
		$$
		we say that $\mathbb W$ and $\mathbb H$ are {\em complementary subgroups} in $\mathbb G$. We write $p=p_{\mathbb W}\cdot p_{\mathbb H}$, denoting with $p_{\mathbb W}$ and $p_{\mathbb H}$ the projections on the two subgroups.
		
		If $\mathbb W$ and $\mathbb H$ are complementary subgroups in a Carnot group $\mathbb G$, then the {\em cone} $C_{\mathbb W,\mathbb H}(q,\alpha)$ of base $\mathbb W$ and axis $\mathbb H$, centered at $q$ and of opening $0\leq\alpha\leq 1$ is defined as
		$$
		C_{\mathbb W,\mathbb H}(q,\alpha) := q\cdot \{p\in\mathbb G:d(p_{\mathbb W},e)\leq \alpha d(p_{\mathbb H},e) \},
		$$
		and $e$ is the identity element of the group $\mathbb G$.
	\end{definition}
	Given $f:\Omega\subseteq \mathbb W\to \mathbb H$, we define the {\em graph} of $f$ as 
	$$
	\mbox{graph}(f):=\{w\cdot f(w): w\in\Omega\} \subseteq \mathbb G.
	$$
	\begin{definition}\label{def:IntrinsicLipschitz}
		Let us assume $\mathbb W$ and $\mathbb H$ are complementary subgroups in a Carnot group $\mathbb G$. Given $f:\Omega \subset \mathbb W\to \mathbb H$, with $\Omega$ open, we say that $f$ is {\em $L$-intrinsic Lipschitz} in $\Omega$, with $L>0$, if 
		$$
		C_{\mathbb W,\mathbb H}\left(p,\frac{1}{L}\right) \cap \mbox{graph}(f)=\{p\}, \qquad \mbox{for all $p\in\mbox{graph}(f)$},
		$$ 
		where $	C_{\mathbb W,\mathbb H}\left(p,\frac{1}{L}\right)$ is defined in \autoref{def:Cone}.
		
	\end{definition}
	\begin{remark}\label{rem:IntrinsicIsContinuous}
		It is not always true that an intrinsic Lipschitz function is Lipschitz in exponential coordinates. Nevertheless an arbitrary intrinsic Lipschitz function is locally H\"older continuous, see \cite[Proposition 2.3.6]{FMS14}.
	\end{remark}
	
	Now we show that the Hausdorff dimension of $C^1_{\rm H}$-hypersurface in a Carnot group of Hausdorff dimension $Q$ is $Q-1$. We state the following implicit function theorem, which is rather classical in the field. In the reference that we give, there is also a general version for intersection of $C^1_{\rm H}$-hypersurfaces.
	\begin{proposition}\label{prop:C1H=iLip}{\rm\cite[Theorem A.5]{DMV19}}
		If $\mathbb G$ is a Carnot group and $S$ is a $C^1_{\rm H}$-hypersurface, let us call $\mathbb W:= T_p^{\rm I}S$, for $p\in S$, as defined in \autoref{Def:TangentGroup} and let  $\mathbb H$ be a homogeneous horizontal subgroup that is a complementary subgroup to $\mathbb W$. 
		
		Then there exists  an open neihgbourhood $U$ of $p$ and an intrisic Lipschitz function $\phi:\mathbb W\to \mathbb H$ such that 
		$$
		{\rm graph}(\phi)\cap U=S\cap U.
		$$
	 	
	\end{proposition}
	\begin{remark}
		By \autoref{Def:TangentGroup} it is always true that $T_p^{\rm I}S$ has a homogeneous horizontal complement $\mathbb H$. Indeed, in exponential coordinates, it suffices to take $\mathbb H := \{(tX_1f(p),\dots,tX_mf(p),0,\dots,0)\}_{t\in\mathbb R}$. 
	\end{remark} 
	The following proposition clarifies what is the Hausdorff dimension of the graph of an intrinsic Lipschitz function.
	\begin{proposition}\label{prop:dimiLip}{\rm \cite[Proposition 2.3.7]{FMS14}}
		Let $\mathbb G$ be a Carnot group and let $\mathbb W$ and $\mathbb H$ be two complementary subgroups in $\mathbb G$. Let us denote $k_{\mathbb W}$ the Hausdorff dimension of $\mathbb W$ and let $\phi:\mathbb W\to \mathbb H$ be an intrinsic $L$-Lipschitz function.
		Then there exist $c_0=c_0(\mathbb G)>0$ and $c_1=c_1(\mathbb W,\mathbb H)>0$ constants such that 
		$$
		\left(\frac{c_0}{1+L}\right)^{k_{\mathbb W}} R^{k_{\mathbb W}} \leq \mathcal{S}^{k_{\mathbb W}} \left({\rm graph}(f)\cap B(p,R)\right) \leq c_1(1+L)^{k_{\mathbb W}}R^{k_{\mathbb W}},
		$$
		for all $p\in{\rm graph}(f)$, $R>0$. 
	\end{proposition}
	\begin{remark}\label{rem:EquivalentDistances}
		From \autoref{prop:dimiLip} we can deduce that the spherical Hausdorff measure $\mathcal{S}^{k_{\mathbb W}}$, restricted to $\mbox{graph}(f)$, is $k_{\mathbb W}$-Alfhors-regular and therefore doubling. We also infer that the Hausdorff dimension of $\mbox{graph}(f)$ is $k_{\mathbb W}$. 
	\end{remark}
	The next basic result can be seen as a consequence of \cite[Theorem 4.4, (iii)]{LDNG19}, where the authors proved a more general result holding for every self-similar metric Lie group. 
	\begin{lemma}\label{lem:HausdorffQ-1}
	In every Carnot group a vertical subgroup of codimension 1 has Hausdorff codimension 1.
	\end{lemma}
	Now we are in position to state and prove the following
	\begin{proposition}\label{prop:C1HisQ-1}
		Let $\mathbb G$ be a Carnot group of Hausdorff dimension $Q$. Let $S$ be a $C^1_{\rm H}$-hypersurface of $\mathbb G$.
		Then the Hausdorff dimension of $S$ is $Q-1$ and $\mathcal{H}^{Q-1}$ restricted to $S$ is a locally doubling measure on $S$.
	\end{proposition}
	\begin{proof}
		In order to prove the theorem we have to take into account \autoref{prop:C1H=iLip} according to which, locally around any $p\in S$, we can write $S$ as the graph of an intrinsic Lipschitz function. After that by using \autoref{prop:dimiLip} together with \autoref{lem:HausdorffQ-1}, we get that the Hausdorff measure $\mathcal{H}^{Q-1}$ with respect to $d$ is locally finite and positive on $S$ and also it is locally $(Q-1)$-Alfhors-regular. From this the conclusion follows immediately.
	\end{proof}
	
	\subsection{The tangent group of a $C^1_{\rm H}$-hypersurface as the Hausdorff tangent}\label{sec:HausdorffTangent}
	In this subsection we remind that the tangent group $T_p^{\rm I}S$ at $p\in S$ of an arbitrary $C^1_{\rm H}$-hypersurface (see \autoref{Def:TangentGroup}) is the Hausdorff tangent at $p$ to $S$.
	This follows from \cite[Theorem 3.1.1]{Koz15} and the identification between the kernel of the differential of a $C^1_{\rm H}$-function $f$ with the tangent of the surface defined by $f$, see \cite[Proposition~2.11]{FSSC03a}.
	 
	The result that we state here is simplified for our aims. For the general statement we refer to \cite[Theorem 3.1.1]{Koz15}. 
	\begin{proposition}\label{prop:GHTangent}{\em [Kozhevnikov]}
		Let $S$ be a $C^1_{\rm H}$-hypersurface in a Carnot group $\mathbb G$. Let us consider $p\in S$ and a function $f$ whose 0-level set coincide locally with $S$, as in \eqref{rep}.
		Then we have that there exists $\beta:(0,+\infty)\to(0,+\infty)$, with $\beta(t)\to0^+$ if $t\to 0^+$, such that for all $r>0$
		$$
		d_H\left(B(p,r)\cap S, B(p,r)\cap (p\cdot T_p^{\rm I}S)\right) \leq \beta(r)r,
		$$
		where $T_p^{\rm I}S$ is defined in \autoref{Def:TangentGroup}.
	\end{proposition}

\begin{remark}\label{rem:GHtangent}
	From the result in \autoref{prop:GHTangent} 
	we eventually get also that if we consider the metric space $(S,d)$, the Gromov-Hausdorff tangent at any point $p\in S$ is (isometric to) $(T_p^{\rm I}S,d)$. 
\end{remark}

	We give here part of the statement of \cite[Theorem 3.3.1]{Koz15}, because it is useful for our aims. For the complete theorem one can see the reference.
	\begin{definition}\label{def:Tan+}{\rm [$\mbox{Tan}^+_{\mathbb G}(S,a)$]}
		Given a subset $S$ of a Carnot group $\mathbb G$ and $a\in S$, we say that $v\in \mathbb G$ is an element of $\mbox{Tan}^+_{\mathbb G}(S,a)$ if there exist a positive sequence $\{t_m\}$ with $t_m\to 0$ and a sequence $\{a_m\}$ of elements of $S$ such that 
		$$
		\lim_{m\to+\infty} d(a_m,a) = 0, \qquad \lim_{m\to+\infty} {\delta}_{1/t_m}(a^{-1}a_m)=v.
		$$
	\end{definition}
	\begin{lemma}\label{lem:Tan+}{\em [Kozhevnikov]}
		Let $S$ be a closed set of a Carnot group $\mathbb G$, and let $a\in S$. If there exists a closed homogeneous set $W$ such that
		$$
		r^{-1}d_H(B(a,r)\cap S, B(a,r)\cap (a\cdot W)) \to 0, \qquad \mbox{when} \quad r\to 0,
		$$
		then 
		$$
		\mbox{Tan}^+_{\mathbb G}(S,a)\subseteq W.
		$$
	\end{lemma}

\subsection{Vertical surfaces}
Now we give the definition of \emph{vertical surface}. Loosely speaking, a vertical surface in a Carnot group $\mathbb G$ is a $C^1$-surface that depends only on the horizontal coordinates.
\begin{definition}\label{def:VerticalSurface}
	Let $\mathbb G$ be a Carnot group identified with $\mathbb R^n$ by means of exponential coordinates. Let $m$ be the dimension of the first stratum of the Lie algebra. 
	A \emph{vertical surface} $V$ is
	$$
	V:=\{x\in\mathbb G\equiv \mathbb R^n: f(x_1,\dots,x_m)=0 \},
	$$
	where $f:\Omega\subseteq \mathbb R^m\to \mathbb R$, with $\Omega$ open, is a $C^1$ function with $\nabla f\neq 0$ on $\{\omega \in \Omega: f(\omega_1,\dots,\omega_m)=0\}$. 
	
	Morevoer, if $f$ is linear we say that $V$ is a \emph{vertical subgroup of codimension 1}.
\end{definition}
\begin{remark}\label{rem:Partial}
	An arbitrary vertical surface as in \autoref{def:VerticalSurface} is a $C^1$-hypersurface with no characteristic points, i.e., points that satisfy \eqref{eqn:CharacteristicPoint}. This is due to the fact that, if $1\leq i\leq m$, in exponential coordinates we have
	$$
	X_i=\partial_{x_i} + r_i(x),
	$$
	where $r_i(x)$ is a polynomial combination of $\partial_{x_{i+1}},\dots,\partial_{x_n}$, see \cite[Proposition 2.2]{FSSC03a}, and then, for all $x\in\omega$,
	$$
	X_if(x) =\partial_{x_i}f(x),
	$$
	as $f$ depends only on the first $m$ variables. Thus, from \autoref{rem:Characteristic}, a vertical surface is also a $C^1_{\rm H}$-hypersurface.
\end{remark}
\begin{remark}
	Every tangent group, as defined in \autoref{Def:TangentGroup}, is a vertical subgroup of codimension 1.
\end{remark}
\begin{lemma}\label{lem:EveryTangent}
	Given a Carnot group $\mathbb G$, there exists a vertical surface $V$ such that for every vertical subgroup $\mathbb W$ of codimension 1 in $\mathbb G$ there exists $p\in V$ such that $T_p^{\rm I}V=W$. 
\end{lemma}
\begin{proof}
	Let us consider 
	$$
	V:= \left\{x\in\mathbb G\equiv \mathbb R^n: \sum_{i=1}^m x_i^2= 1\right\}.
	$$
	At an arbitrary point $p=(x_1,\dots,x_m,x_{m+1}\dots,x_n)$,we have that, by \autoref{Def:TangentGroup} and \autoref{rem:Partial},
	$$
	T_p^{\rm I}V=\left\{v\in\mathbb G\equiv\mathbb R^n: \sum_{i=1}^m v_ix_i=0\right\}
	$$
	and then,  as any linear function $f:\mathbb R^m \to \mathbb R$ can be written as $f(v)=\sum_{i=1}^m v_ix_i$ for a vector $(x_1,\dots,x_m)$ of norm 1, we get the desired conclusion. 
\end{proof}

\section{Notions of rectifiability}\label{sect:Rectifiability}
In this section we introduce several kinds of notions of rectifiability. First we say, in \autoref{def:Grect}, what it means for a metric space $(X,d)$, to be $(\mathcal{F},\mu)$-rectifiable, where $\mathcal{F}$ is a family of metric spaces and $\mu$ is an outer measure on $X$. Namely $(X,d)$ is $(\mathcal{F},\mu)$-rectifiable iff it is $\mu$-a.e.~covered by bi-Lipschitz images of subsets of metric spaces in $\mathcal{F}$. 

Then we specialize this notion in \autoref{def:TRect} by taking the class of metric spaces that are proper, locally doubling and with bi-Lipschitz equivalent tangents. In \autoref{rem:TangRectIsCarnotRect} we discuss further specializations of this notion.

Then we give the notion of Pauls Carnot rectifiability in \autoref{def:LipCarRect}, generalizing the definition given in \cite[Definition 4.1]{Pau04}. In \autoref{rem:LipCounterpart} we briefly discuss the Lipschitz variant of $(\mathcal{F},\mu)$-rectifiability for specific families $\mathcal{F}$.

\begin{definition}[$(\mathcal{F},\mu)$-rectifiability]\label{def:Grect}
	Given a family $\mathcal{F}$ of metric spaces we say that a metric space $(X,d)$, with an outer measure $\mu$ on it, is \emph{$(\mathcal{F},\mu)$-rectifiable} if there exist countably many bi-Lipischitz embeddings $f_i:U_i\subseteq (X_i,d_i)\to (X,d)$ where $
	(X_i,d_i)\in\mathcal{F}$, $i\in\mathbb N$, and 
	$$
	\mu\left(X\setminus\bigcup_{i\in \mathbb N} f_i(U_i)\right) = 0.
	$$
	We say that a metric space $(X,d)$ is \emph{purely} $(\mathcal{F},\mu)$-\emph{unrectifiable} if for every $(X',d')\in\mathcal{F}$ and every bi-Lipschitz embedding $f:U\subseteq (X',d')\to (X,d)$ it holds 
	$$
	\mu(f(U))=0.
	$$
\end{definition}
\begin{remark}\label{rem:LocComp}
	We notice that if $(X,d)$ is complete we can equivalently ask each set $U_i$ to be closed in \autoref{def:Grect}. Indeed, in this case every bi-Lipschitz map $f_i:U_i\to (X,d)$ extends to a bi-Lipschitz map $\overline{f_i}:\overline{U_i}\to(X,d)$.
\end{remark}
\begin{remark}\label{rem:Purely}
	Having a look at \autoref{def:Grect}, assuming we have $\mu(X)>0$, which will be always in our case, we see that one necessary condition for the $(\mathcal{F},\mu)$-rectifiability of $(X,d)$ is the existence of at least one bi-Lipschitz map $f:U\subseteq (X',d')\to (X,d)$, where  $(X',d')\in\mathcal{F}$ and $\mu(f(U))>0$. So if a metric space $(X,d)$, with an outer measure $\mu$ on it such that $\mu(X)>0$, is $(\mathcal{F},\mu)$-purely unrectifiable then it cannot be $(\mathcal{F},\mu)$-rectifiable.
\end{remark}
Specializing the family $\mathcal{F}$ and $\mu$ in \autoref{def:Grect}, we can give the following definitions.
\begin{definition}[bi-Lipschitz homogeneous rectifiability]\label{def:TRect}
	Let $(X,d)$ be a metric space of Hausdorff dimension $k$. Set $\mathcal{T}_k:= \{(X_i,d_i)\}_{i\in I}$ to be the family of all the metric spaces $(X_i,d_i)$ such that:
	\begin{itemize} 
	\item $(X_i,d_i,\mathcal{H}^k)$ is a proper locally doubling metric measure space, with $k=\dim_H X_i$;
	\item any two tangent spaces, at any two points of $X_i$, are bi-Lipschitz equivalent.
	\end{itemize}
	We say that $(X,d)$ is \emph{bi-Lipschitz homogeneous rectifiable} if it is \emph{$(\mathcal{T}_k,\mathcal{H}^k)$}-rectifiable according to \autoref{def:Grect}. We say that $(X,d)$ is \emph{purely bi-Lipschitz homogeneous unrectifiable} if it is purely $(\mathcal{T}_k,\mathcal{H}^k)$-unrectifiable according to \autoref{def:Grect}.
\end{definition}
\begin{remark}\label{rem:TangRectIsCarnotRect}
	The family $\mathcal{T}_k$ defined in \autoref{def:TRect} is very rich. For example it contains all homogeneous Lie groups $\mathbb G$ equipped with left-invariant homogeneous distances, with Hausdorff dimension $k$. Indeed, every tangent space at any point of such a group $\mathbb G$ is isometric to $\mathbb G$ and $(\mathbb G,d_{\mathbb G},\mathcal{H}^k)$ is proper because it is locally compact and admits dilations, and it is locally doubling because it is Ahlfors-regular \cite[Theorem 4.4, (iii)]{LDNG19}. We remark here that the larger class of self-similar metric Lie groups of Hausdorff dimension $k$, whose definition is in \cite{LDNG19}, is still a subclass of $\mathcal{T}_k$. Going beyond Lie groups, we remark that in $\mathcal{T}_k$ one has all those Carnot-Carathéodory spaces whose nilpotentization is constantly equal to a stratified group of homogeneous dimension $k$. This last statement is a consequence of Mitchell's theorem (see \cite{Mit85} and \cite{Bel96}) and the bi-Lipschitz equivalence of left-invariant homogeneous distances.
	
	In the very rich class of homogeneous Lie groups we distinguish homogeneous subgroups of Hausdorff dimension $k$ of arbitrary Carnot groups, with the restricted distance, and obviously also Carnot groups of Hausdorff dimension $k$. We can then give different notions of rectifiability for each of these subfamilies of $\mathcal{T}_k$.
	
	Notice that if we take the subfamily of $\mathcal{T}_k$ made of arbitrary homogeneous subgroups, of dimension $k$, of Carnot groups, we obtain a variation of \cite[Definition 3]{CP04} where we now allow countably many homogeneous subgroups but we require bi-Lipschitz maps. Similarly, if we only consider Carnot groups, we obtain a similar variation of \cite[Definition 4.1]{Pau04}.
\end{remark}
We give next a criterion for purely bi-Lipschitz homogeneous unrectifiability. 
\begin{lemma}\label{lem:KeyLemma}
	Let $(X,d,\mathcal{H}^k)$ be a proper locally doubling metric measure space, with $k=\dim_H X$. If every $\mathcal{H}^k$-positive measure subset of $X$ contains two points that have two tangent spaces that are not bi-Lipschitz equivalent, then $(X,d)$ is purely bi-Lipschitz homogeneous unrectifiable (according to \autoref{def:TRect}).
\end{lemma}
\begin{proof}
	We prove that there is no bi-Lipschitz map $f:U\subseteq \left(X',d'\right)\to (X,d)$, where $\mathcal{H}^k(f(U))>0$ and $\left(X',d'\right)\in\mathcal{T}_k$. We restrict ourselves to consider $U$ closed, see \autoref{rem:LocComp}.
	
	If there exists such a map, first of all notice that $\mathcal{H}^k(U)>0$ because $f$ is bi-Lipschitz. Now we can restrict ourselves to the points of density of $U$ with respect to $\mathcal{H}^k$, say $W$, and $W$ is a set of full $\mathcal{H}^k$-measure in $U$ \cite[Corollary 2.13, Theorem 3.1 and Remark 3.3]{Rig18}. Then, by the fact that $f$ is bi-Lipschitz, the set $f(W)$ has full $\mathcal{H}^k$-measure in $f(U)$. The set $Z$ of points in $f(W)$ of density of $f(U)$ with respect to $\mathcal{H}^k$, is still a set of full $\mathcal{H}^k$-measure in $f(U)$ because it is the intersection of two sets of full $\mathcal{H}^k$-measure in $f(U)$. Then it holds $\mathcal{H}^k(Z)>0$ since $\mathcal{H}^k(f(U))>0$.
	
	By hypothesis there exist two points $x,y\in W$ and $p=f(x),q=f(y)\in Z$ with two non-bi-Lipschitz tangent spaces $T_p$ and $T_q$. Because of the fact that we are dealing with $1$-density points, we can say that $\mbox{Tan}\left(U,x,d'\right)=\mbox{Tan}\left(X',x,d'\right)$ and $\mbox{Tan}(f(U),p,d)=\mbox{Tan}(X,p,d)$ and the same holds with $y$ and $q$, see \cite[Proposition 3.1]{LD11}. Passing to the tangents in $p$ and $x$ we get, as in \cite[Section 5.2]{LDY19}, some induced bi-Lipschitz map between $T_p$ and one element of $\mbox{Tan}\left(X',x,d'\right)$. In the same way we get a bi-Lipschitz map between $T_q$ and one element of $\mbox{Tan}\left(X',y,d'\right)$. By hypothesis each element of $\mbox{Tan}\left(X',x,d'\right)$ is bi-Lipschitz to each element of $\mbox{Tan}\left(X',y,d'\right)$, so that at the end $T_p$ is bi-Lipschitz to $T_q$, which is a contradiction. 
\end{proof}

Let us point out that in \autoref{def:Grect} we require the parametrizing maps to be bi-Lipschitz while for the classical definitions of rectifiability one may just ask for the map to be Lipschitz. We next give the Lipschitz counterpart of \autoref{def:Grect} for the family of Carnot groups.
\begin{definition}[Pauls Carnot rectifiability]\label{def:LipCarRect}
	Let $(X,d)$ be a metric space of Hausdorff dimension $k$.
	We say that $(X,d)$ is \emph{Pauls Carnot rectifiable} if there exist  countably many Carnot groups $\mathbb G_i$ of Hausdorff dimension $k$ and Lipschitz maps $f_i:U_i\subseteq (\mathbb G_i,d_i)\to (X,d)$ such that
	$$
	\mathcal{H}^k\left(X\setminus \bigcup_{i\in\mathbb N} f_i(U_i)\right)=0.
	$$
	We say that $(X,d)$ is \emph{purely Pauls Carnot unrectifiable} if for every Carnot group $\mathbb G$ of  Hausdorff dimension $k$, every Lipschitz map $f:U\subseteq (\mathbb G,d_{\mathbb G})\to (X,d)$ satisfies
	$$
	\mathcal{H}^k(f(U))=0.
	$$
\end{definition}
\begin{remark}\label{rem:CountLip}
	The definition given in \autoref{def:LipCarRect} is a generalization of \cite[Definition 4.1.]{Pau04} where it was considered only one Carnot group for the parametrization of $X$. The definition of {\em purely $\mathbb G$-unrectifiability}, with a Carnot group $\mathbb G$, was already given in \cite[Definition 3.1]{Mag04}. That is, given a Carnot group $\mathbb G$ of Hausdorff dimension $k$, we say that a metric space $(X,d)$ is {\em purely $\mathbb{G}$-unrectifiable} if every Lipschitz map $f:U\subseteq \mathbb G\to X$ satisfies $\mathcal{H}^k(f(U))=0$.
\end{remark}
\begin{remark}\label{rem:LipCounterpart}
	In this paper we will not focus on the Lipschitz counterpart of \autoref{def:TRect}. The case of homogeneous subgroups of Carnot groups would lead to to a variant of \cite[Definition 3]{CP04} allowing countably many possibly different groups. We think there are pathological examples and more easy-to-ask questions that we are not able to answer up to now. 
	
	For example Peano's curve tells that the Euclidean plane $\mathbb R^2$ can be Lipschitz rectified with $\left(\mathbb R,\|\cdot\|^{1/2}\right)$. Notice that $\left(\mathbb R,\|\cdot\|^{1/2}\right)$ is the vertical line in the Heisenberg group.\footnote{We have evidence that every Carnot group of Hausdorff dimension $Q$ can be Lipschitz rectified with $\left(\mathbb R,\|\cdot\|^{1/Q}\right)$, which is a subgroups of every Carnot group of step $Q$.}
	
	Forcing the topological dimension to be the same, we also wonder whether there exists a Lipschitz map 
	$$
	f:U\subseteq \left(\mathbb R^3,\|\cdot\|^{3/4}\right) \to \mathbb H^1,
	$$
	with $\mathcal{H}^4(f(U))>0$.
\end{remark}
\begin{remark}\label{rem:Purely2}
	As in \autoref{rem:Purely}, if $(X,d)$ has Hausdorff dimension $k$ and $\mathcal{H}^k(X)>0$, it holds that if $(X,d)$ is purely Pauls Carnot unrectifiable then it is not Pauls Carnot rectifiable. 
\end{remark}

\section{A Carnot algebra with uncountable non-isomoprhic Carnot subalgebras}\label{sec:Construction}
In this section we prove that there exists a Carnot algebra $\Gg$ of dimension 8 that has uncountably many pairwise non-isomorphic Carnot subalgebras of dimension 7. The Lie algebra $\Gg$ is constructed in \autoref{def:G} and in \autoref{prop:R1} we prove the result. 
\begin{definition}\label{def:Gmu}
	Given $\mu\in\mathbb R$, we call $\Gg_{\mu}$ the Carnot algebra of step 3 and dimension 7 given by
	$$
	\Gg_{\mu}=\Vv^1_{\mu}\p\Vv^2_{\mu}\p\Vv^3_{\mu},
	$$
	where $$\Vv^1_{\mu}=\mbox{span}\{X_1,X_2,X_3\}, \quad \Vv^2_{\mu}=\mbox{span}\{X_4,X_5,X_6\}, \quad \Vv^3_{\mu}=\mbox{span}\{X_7\},
	$$
	with the following relations 
	\begin{equation}
	\begin{aligned}
	&[X_1,X_2]=X_4, \quad &[X_1,X_3]=-X_6, \quad &[X_2,X_3]=X_5; \\
	&[X_1,X_5]=-X_7, \quad &[X_2,X_6]=\mu X_7, \quad &[X_3,X_4]=(1-\mu)X_7,
	\end{aligned}
	\end{equation}
	where all the other commutators between two vectors of the basis $\{X_1,\dots,X_7\}$ that are not listed above are zero.
\end{definition}
\begin{remark}\label{rem:Unc}
	The family $\{\Gg_{\mu}\}_{\mu\in\mathbb R}$ in \autoref{def:Gmu} consists of uncountably many pairwise non-isomorphic Carnot algebras, which are called of type 147E, see \cite{Gong98}. Indeed, if $\mu_1,\mu_2\notin\{0,1\}$, the Lie algebra $\Gg_{\mu_1}$ is isomorphic to $\Gg_{\mu_2}$ if and only if $I(\mu_1)=I(\mu_2)$, where 
	$$
	I(\mu):= \frac{(1-\mu+\mu^2)^3}{\mu^2(\mu-1)^2}.
	$$
\end{remark}

Our plan is to next add a direction $X$ in the first stratum of a specific Carnot algebra given by \autoref{def:Gmu}, namely the one with $\mu=0$. We show the existence of uncountably many pairwise non-isomorphic Carnot subalgebras of dimension 7 in this new Carnot algebra of dimension 8.

\begin{definition}\label{def:G}
	We call $\Gg$ the Carnot algebra of step 3 and dimension 8 given by 
	$$
	\Gg=\Vv^1\p\Vv^2\p\Vv^3,
	$$
	where
	$$
	\Vv^1=\mbox{span}\{X_0,X_1,X_2,X_3\}, \quad \Vv^2=\mbox{span}\{X_4,X_5,X_6\}, \quad \Vv^3=\mbox{span}\{X_7\},
	$$
	with the following bracket relations 
	\begin{equation}\label{relationsG}
	\begin{aligned}
	&[X_1,X_2]=X_4, \quad  [X_1,X_3]=-X_6, \quad [X_1,X_0]=-X_4, \quad [X_2,X_3]=X_5; \\
	&[X_1,X_5]=-X_7, \quad [X_3,X_4]=X_7, \quad [X_0,X_6]=X_7,
	\end{aligned}
	\end{equation}
	and all the other commutators between two elements of the basis $\{X_0,X_1,\dots,X_7\}$ that are not listed above are 0.
\end{definition}
\begin{remark}
	Let us show that the one defined in \autoref{def:G} is a Lie algebra. It suffices to verify Jacobi identity on triples of pairwise different vectors of the basis. Being the step of the stratification equal to $3$, it suffices to show the Jacobi identity on vectors in the first stratum $\Vv^1$. Then, as we are extending $\Gg_0$ in \autoref{def:Gmu}, we just have to check the Jacobi identity on the triples $\{X_1,X_2,X_0\}$, $\{X_2,X_3,X_0\}$ and $\{X_1,X_3,X_0\}$. A simple computation yields 
	\begin{equation}
	\begin{aligned}
	&[X_1,[X_2,X_0]]+[X_2,[X_0,X_1]]+[X_0,[X_1,X_2]]=0+[X_2,X_4]+[X_0,X_4]=0, \\
	&[X_2,[X_3,X_0]]+[X_3,[X_0,X_2]]+[X_0,[X_2,X_3]]=0+0+[X_0,X_5]=0,\\
	&
	[X_1,[X_3,X_0]]+[X_3,[X_0,X_1]]+[X_0,[X_1,X_3]]=0+[X_3,X_4]-[X_0,X_6]=X_7-X-7=0,
	\end{aligned}
	\end{equation}
	which is what we want. 
\end{remark}
Now we are ready for the following proposition.
\begin{proposition}\label{prop:R1}
	If $\Gg$ is the Carnot algebra of dimension 8 and step 3 in \autoref{def:G}, then there exist uncountably many Carnot subalgebras  of dimension 7 of $\Gg$ that are pairwise non-isomorphic. 
\end{proposition}
\begin{proof}
	We present explicitely an uncountable family of Carnot subalgebras of dimension 7 of $\Gg$, indexed by $\lambda\in\mathbb R$, which are isomorphic to $\Gg_{\lambda}$ in \autoref{def:Gmu} if $\lambda\neq 1$. Then by \autoref{rem:Unc} we get the conclusion. 
	
	Given $\lambda\in \mathbb R$, with $\lambda\neq 1$, let us define the following vector in $\Vv^1\subseteq \Gg$,
	\begin{equation}\label{DefY}
	\quad Y_2:= X_2+\lambda X_0.
	\end{equation}
	Then $\{X_1,Y_2,X_3\}$ are linearly independent vectors of $\Vv^1$.
	By explicit computations, using the relations in \eqref{relationsG}, we have
	\begin{equation}\label{1Layerh}
	\begin{aligned}
	&[X_1,Y_2]=(1-\lambda)X_4=: Y_4, \\
	&[X_1,X_3]=-X_6, \\
	&[Y_2,X_3]=X_5; \\
	\end{aligned}
	\end{equation}
	\begin{equation}\label{2Layerh}
	\begin{aligned}
	&[X_1,Y_4]=0, \\
	&[X_1,X_5]=-X_7, \\
	&[X_1,X_6]=0, \\
	&[Y_2,Y_4]=0, \\
	&[Y_2,X_5]=0, \\
	&[Y_2,X_6]=\lambda X_7, \\
	&[X_3,Y_4]=(1-\lambda)X_7, \\
	&[X_3,X_5]=0, \\
	&[X_3,X_6]=0,
	\end{aligned}
	\end{equation}
	and all the other commutators beetwen two elements of the linearly independent vectors $\{X_1,Y_2,X_3,Y_4,X_5,X_6,X_7\}$, that are not listed above, vanish. Then in view of \eqref{1Layerh} and \eqref{2Layerh}, if $\lambda\neq 1$, the subspace $\Gw^1:= \mbox{span}\{X_1,Y_2,X_3\}$ generates a Carnot subalgebra  of step 3 and dimension $7$ in $\Gg$, which is isomorphic to $\Gg_{\lambda}$ in \autoref{def:Gmu}.
\end{proof}

	\section{Main results}\label{sec:MainResults}
	In this section we construct the example that satisfies \autoref{thm:T1.5Intro} and we prove the properties discussed in the introduction. We build the hypersurface $S$ in the Carnot group $\mathbb G$ whose Lie algebra $\Gg$ is as in \autoref{def:G}.
	
	First of all let us identify $\mathbb G$ with $\mathbb R^8$ by using exponential coordinates and the ordered basis $\left(X_0,X_1,\dots,X_7\right)$
	\begin{equation}\label{exp1stKind}
	\begin{aligned}
	x&=(x_0,x_1,x_2,x_3,x_4,x_5,x_6,x_7)\to \\ &\to \exp\left(x_0X_0+x_1X_1+x_2X_2+x_3X_3+x_4X_4+x_5X_5+x_6X_6+x_7X_7\right).
	\end{aligned}
	\end{equation}
	In these coordinates we can express the left-invariant vector fields $X_0(x),X_1(x),X_2(x),X_3(x)$ that extend $X_0,X_1,X_2,X_3$, in this way, see \cite[Proposition 2.2]{FSSC03a}:
	\begin{equation}\label{ExpressionsOfVector}
	\begin{aligned}
	&X_0(x)=\partial_{x_0}+r_0(x), \\
	&X_1(x)=\partial_{x_1}+r_1(x), \\
	&X_2(x)=\partial_{x_2}+r_2(x), \\
	&X_3(x)=\partial_{x_3}+r_3(x),
	\end{aligned}
	\end{equation}
	where $r_0(x),r_1(x),r_2(x),r_3(x)$ are combinations, with polynomial coefficients of the coordinates, of $\partial_{x_4},\partial_{x_5},\partial_{x_6},\partial_{x_7}$. 
	Now we are ready to state and prove one of the main results of this article.
	\begin{proposition}\label{prop:R2}
	  There exist a Carnot group $\mathbb G$ and a $C^{\infty}$ non-characteristic hypersurface $S\subseteq \mathbb G$ with uncountably many pairwise non-isomorphic tangent groups.
	\end{proposition}
	\begin{proof} 
	Let us consider the Carnot algebra $\Gg$ in \autoref{def:G} and $\mathbb G:= \exp\Gg$ identified with $\mathbb R^8$ by means of the exponential coordinates in \eqref{exp1stKind}.
	
	Let us consider the family of vertical subgroups of codimension 1 in $\mathbb G$ given by 
	$$
	\mathbb G_{\lambda} := \{v\in\mathbb G\equiv \mathbb R^8: \lambda v_2-v_0=0\}.
	$$
	The Lie algebra of $\mathbb G_{\lambda}$ is isomorphic to the algebra $\Gg_{\lambda}$ if $\lambda\neq 1$ according to the proof of \autoref{prop:R1}. Then the family $\{\mathbb G_{\lambda}\}_{\lambda \in \mathbb R}$ contains uncountably many non-isomorphic Carnot groups and the conclusion follows taking the vertical surface in $\mathbb G$ given by \autoref{lem:EveryTangent} that is smooth and non-characteristic due to \autoref{rem:Partial}.
\end{proof}
	\begin{remark}
		In particular, every $S$ as in \autoref{prop:R2} is not bi-Lipschitz equivalent to an open set in a Carnot group. This follows from a blow-up argument and Pansu's differentiability theorem \cite{Pan89}. The argument will be made clear in the proof of the forthcoming \autoref{thm:T1}. We stress that even for some sub-Riemannian manifolds the constance of the tangent may not give bi-Lipschitz local equivalence with the tangent Carnot group, see \cite{LOW14}.
	\end{remark}
	\begin{remark}\label{rem:Ex2}
	We give another example of smooth non-characteristic hypersurface satisfying \autoref{prop:R2}. The particular form of this example will help us in showing the forthcoming \autoref{thm:T1} and \autoref{thm:T1.5}.
	Let us consider again the Carnot algebra $\Gg$ in \autoref{def:G} and $\mathbb G:= \exp\Gg$ identified with $\mathbb R^8$ by means of the exponential coordinates in \eqref{exp1stKind}. Let us consider the vertical surface
	\begin{equation}\label{defS}
	S=\left\{x\in\mathbb G\equiv \mathbb R^8: f(x):= \frac{1}{3}x_2^3+x_0 = 0 \right\}.
	\end{equation}
 	By \autoref{rem:Partial} this is a smooth non-characteristic hypersurface. 
	From easy computations due to the particular form of $X_i$'s in \eqref{ExpressionsOfVector} and from the definition of tangent group in  \autoref{Def:TangentGroup}, it follows that 
	$$
	{\rm Lie}(T_x^{\rm I}S)=
	\mbox{span}\{X_1,X_2-x_2^2X_0,X_3,X_4,X_5,X_6,X_7\},
	$$
	and then ${\rm Lie}(T_x^{\rm I}S)$ is isomorphic to the Carnot algebra generated by $\Gw^1$ defined in the proof of \autoref{prop:R1}, where the $\lambda$ there is now equal to $-x_2^2$. Then, the Lie algebra  ${\rm Lie}(T_x^{\rm I}S)$ is isomorphic to $\Gg_{-x_2^2}$ defined in \autoref{def:Gmu}.
	Because of the fact that given any $\lambda\leq 0$ there is always a point in $S$ satisfying $\lambda=-x_2^2$, \autoref{rem:Unc} grants us that the family $\{{\rm Lie}(T_x^{\rm I}S)\}_{x\in S}$ contains uncountably many pairwise non-isomorphic Carnot algebras and then the family $\{T_x^{\rm I}S\}_{x\in S}$ contains uncountably many pairwise non-isomorphic Carnot groups.
	\end{remark}
	\begin{theorem}\label{thm:T1}
		There exist a Carnot group $\mathbb G$ and a $C^{\infty}$ non-characteristic hypersurface $S\subseteq \mathbb G$, of Hausdorff dimension 12, such that on every $\mathcal{H}^{12}$-positive measure subset of it there are two points with non-isomorphic tangents. In particular, the set $S$ is purely bi-Lipschitz homogeneous unrectifiable according to \autoref{def:TRect}.
	\end{theorem}
	\begin{proof}
		Let us consider $\Gg$ the Carnot algebra in \autoref{def:G}. Let us identify $\mathbb G:= \exp\Gg$ with $\mathbb R^8$ by means of the exponential coordinates in \eqref{exp1stKind} and let us fix a left-invariant homogeneous distance $d$ on $\mathbb G$. Let us consider $S$ as in \autoref{rem:Ex2}. By \autoref{prop:C1HisQ-1}, the definition of $S$, and the fact that the Hausdorff dimension of $\mathbb G$ is $13$, we get that the Hausdorff dimension of $S$ is $12$ and $(S,d,\mathcal{H}^{12})$ is a proper locally doubling metric measure space. Notice also that by \autoref{prop:GHTangent} (see also \autoref{rem:GHtangent}) we get that the Gromov-Hausdorff tangent at each point $x\in S$ is unique and isometric to $(T_x^{\rm I}S,d)$. Notice that for all $x\in S$, the space $T_x^{\rm I}S$ is a Carnot group, as it is shown in the construction of $S$ given in \autoref{rem:Ex2}.
		
		We claim that 
		\begin{equation}\label{eqn:Slambda} \mathcal{H}^{12}(S\cap\{x_2=\lambda\})=0, \qquad \forall\lambda\in\mathbb R.
		\end{equation}
		Indeed, we know that $S\cap\{x_2=\lambda\}$ is the intersection of two $C^1_{\rm H}$-hypersurfaces that satisfy the improved version of \autoref{prop:C1H=iLip} contained in \cite[Theorem A.5]{DMV19}. Then we get that $S\cap\{x_2=\lambda\}$ is locally the graph of an intrinsic Lipschitz function defined on 
		$$
		\mathbb W :=\{v\in \mathbb G\equiv \mathbb R^8: x_2^2v_2+v_0=0\}\cap\{v\in \mathbb G\equiv \mathbb R^8: v_2=0\}=\{v\in\mathbb G\equiv \mathbb R^8:v_0=v_2=0\},
		$$
		with values in the horizontal subgroup $\mathbb H:=\{\exp(tX_0+sX_2): t,s\in\mathbb R\}$. Notice that $\mathbb H$ is a subgroup because $[X_0,X_2]=0$.  By using \cite[Theorem 4.4, (iii)]{LDNG19} we get that $\mathbb W$ has Hausdorff dimension 11 with respect to the distance $d$ and then by \autoref{prop:dimiLip} we get \eqref{eqn:Slambda}.
		
		Now we claim that each subset $U$ of $S$ that satisfies $\mathcal{H}^{12}(U)>0$ has at least two points with two non-bi-Lipschitz Gromov-Hausdorff tangents. Indeed, the equation \eqref{eqn:Slambda} tells us that for each $U\subseteq S$ with $\mathcal{H}^{12}(U)>0$, the coordinate function $x_2$ takes on $U$ uncountable values. This, according to the fact that $T_x^{\rm I}S$ is a Carnot group isomorphic to the one with Lie algebra $\Gg_{-x_2^2}$ (see \autoref{rem:Ex2}), immediately tells that there are in $U$ at least two points with two non-isomorphic (because of \autoref{rem:Unc}) Carnot groups as tangent. By Pansu's theorem \cite{Pan89}, two non-isomorphic Carnot groups cannot be bi-Lipschitz equivalent, so the claim follows. Now the proof is completed by using the criterion shown in \autoref{lem:KeyLemma}.
	\end{proof}
	From \autoref{rem:Purely} we have the following consequence to \autoref{thm:T1}.
	\begin{corollary}\label{cor:T1}
	There exist a Carnot group $\mathbb G$ and a $C^{\infty}$ non-characteristic hypersurface $S\subseteq \mathbb G$ that it is not bi-Lipschitz homogeneous rectifiable according to \autoref{def:TRect}
	\end{corollary}
	\begin{remark}\label{rem:NotAble}
	Notice that from \autoref{cor:T1} it follows that $S$ is not rectifiable
	according to the countable bi-Lipschitz variant of \cite[Definition 3]{CP04}, see \autoref{rem:TangRectIsCarnotRect}. We notice here that we still are not able to prove that our counterexample is not rectifiable according to \cite[Definition 3]{CP04}, see \autoref{rem:LipCounterpart}. Nevertheless, in the forthcoming \autoref{cor:T1.5} we show that it is not so according to \cite[Definition 4.1]{Pau04}.
	\end{remark}
	\begin{theorem}\label{cor:T1.5}
		There exist a Carnot group $\mathbb G$ and a $C^{\infty}$ non-characteristic hypersurface $S\subseteq \mathbb G$ that is purely Pauls Carnot unrectifiable according to \autoref{def:LipCarRect}.
	\end{theorem}
	\begin{proof}
	
	Let us take $S$ and $\mathbb G$ as in \autoref{rem:Ex2}. Let us fix on $\mathbb G$ a homogeneous left-invariant distance $d$. Then from \autoref{prop:C1HisQ-1} we get that the Hausdorff dimension of $S$ is 12, because the Hausdorff dimension of $\mathbb G$ is 13. We will show there is no Lipschitz map $f:U\subseteq \hat{\mathbb G}\to (S,d)$, with $\hat{\mathbb G}$ a Carnot group, $\dim_H\hat{\mathbb G}=12$ and $\mathcal{H}^{12}(f(U))>0$. 
	
	Suppose by contradiction there is such a map. We can assume $U$ closed, see \autoref{rem:LocComp}. By composing the map $f$ with the inclusion $i:S\hookrightarrow\mathbb G$ we get a Lipschitz map $\tilde{f}:U\subseteq \hat{\mathbb G} \to \mathbb G$. We will make use of results and notation in Section \ref{sect:Preliminaries}.
	
	Let us call $U_{\rm ND}\subseteq U$ the set of points where $\tilde{f}$ is non-differentiable, $U_{\rm I}\subseteq U$ the set of differentiability points $x$ of $\tilde{f}$ for which $D\tilde{f}_x:\hat{\mathbb G}\to\mathbb G$ is injective and $U_{\rm NI}\subseteq U$ the set of  differentiability points $x$ of $\tilde{f}$ for which $D\tilde{f}_x$ is not injective. We thus have $U=U_{\rm ND}\sqcup U_{\rm I}\sqcup U_{\rm NI}$, $\tilde{f}(U)=\tilde{f}(U_{\rm ND})\cup \tilde{f}(U_{\rm I})\cup \tilde{f}(U_{\rm NI})$ and we know, from Rademacher theorem (\autoref{thm:Rademacher}) and the fact that $\tilde{f}$ is Lipschitz, that $\mathcal{H}^{12}(\tilde{f}(U_{\rm ND}))=\mathcal{H}^{12}(U_{\rm ND})=0$.
	
	We claim that $\mathcal{H}^{12}(\tilde{f}(U_{\rm I}))>0$. 
	Indeed, the Hausdorff dimension of $\hat{\mathbb G}$ is 12. Thus, for $x\in U_{\rm NI}$, we get that $D\tilde{f}_x(\hat{\mathbb G})$ is a homogeneous subgroup of $\mathbb G$ of Hausdorff dimension at most 11, see \autoref{lem:StrictDimension} below. Then
	$$
	J_{12}(D\tilde{f}_x)=\frac{\mathcal{H}^{12}(D\tilde{f}_x(B(0,1)))}{\mathcal{H}^{12}(B(0,1))} = 0,
	$$
	and from \autoref{thm:AreaFormula}
	applied to $\tilde{f}:U_{\rm NI}\to\mathbb G$ we get $\mathcal{H}^{12}(\tilde{f}(U_{\rm NI}))=0$. Now we conclude the proof of the claim:
	$$
	\begin{aligned}
	\mathcal{H}^{12}(\tilde{f}(U_{\rm I}))&=\mathcal{H}^{12}(\tilde{f}(U_{\rm I}))+\mathcal{H}^{12}(\tilde{f}(U_{\rm NI}))+\mathcal{H}^{12}(\tilde{f}(U_{\rm ND}))\\
	&\geq \mathcal{H}^{12}(\tilde{f}(U))>0.
	\end{aligned}
	$$
	
	For every point $z$ in $U_{\rm I}$ there exists an injective Carnot homomorphism $D\tilde{f}_z:\hat{\mathbb G}\to\mathbb G$. For how it is constructed the differential $D\tilde{f}_z$ (see \autoref{rem:DefinitionOfDifferential}) we know that for $\omega$ in a dense subset $\Omega$ of $\hat{\mathbb G}$ we have 
	$$
	D\tilde {f}_z(\omega)=\lim_{z\delta_t\omega, t\to 0} \delta_{1/t}\left(\tilde{f}(z)^{-1}\tilde{f}(z\delta_t\omega)\right).
	$$
	From $\tilde{f}(z\delta_t\omega) \in S$ and $\tilde{f}(z\delta_t\omega)\to \tilde{f}(z)$ we thus get that $D\tilde{f}_z(\omega) \in \mbox{Tan}^+_{\mathbb G}(S,\tilde{f}(z))$ (see \autoref{def:Tan+}). Then
	thanks to \autoref{lem:Tan+}, applied with $W=T_{\tilde{f}(z)}^{\rm I} S$ in view of \autoref{prop:GHTangent}, we get that $D\tilde{f}_z(\omega)$ takes values in $T_{\tilde{f}(z)}^{\rm I}S$ for $\omega \in \Omega$. 
	Now taking into account that $D\tilde{f}_z$ is defined on all of $\hat{\mathbb G}$ by density (see \autoref{rem:DefinitionOfDifferential}) and considering that $T_{\tilde{f}(z)}^{\rm I}S$ is closed, we get that $D\tilde{f}_z$ takes values in $T_{\tilde{f}(z)}^{\rm I}S$, which is a Carnot subgroup of $\mathbb G$ of Hausdorff dimension $12$, thanks to the explicit expression of the tangent in \autoref{rem:Ex2} and \cite[Theorem 4.4, (iii)]{LDNG19}. Thus as $\hat{\mathbb G}$ has Hausdorff dimension 12 itself and $D\tilde{f}_z$ is injective, we get that $D\tilde{f}_z$ is an isomorphism and so $\hat{\mathbb G}$ is isomorphic to $T_{\tilde{f}(z)}^{\rm I}S$ for every $z\in U_{\rm I}$. Thus we get a contradiction with \autoref{thm:T1} because we proved $\mathcal{H}^{12}(\tilde{f}(U_I))>0$.
	\end{proof}
	\begin{lemma}\label{lem:StrictDimension}
		Let $\phi:\mathbb G\to \mathbb H$ be a Carnot homomorphism between two Carnot groups. If $\phi$ is not injective then it holds
		$$
		\dim_H \mathbb G \geq \dim_H \phi(\mathbb G)+1.
		$$
	\end{lemma}
	\begin{proof}
		By definition of Carnot homomorphism we get that $\mbox{Ker}\phi$ is a homogeneous subgroup of $\mathbb G$ and $\phi(\mathbb G)$ is a homogeneous subgroup of $\mathbb H$.
		If an element $g$ is in $\Vv_i$ for some $i$, we say that $i$ is the degree of $g$ and write $\deg g=i$.
		We take $\{e_1,\dots,e_l,e_{l+1},\dots,e_n\}\subseteq \cup_{i=1}^s \Vv_i$ a basis of $\Gg$, such that $\{e_1,\dots,e_l\}$ is a basis of $\mbox{Ker}\phi$. Then $\{ \phi_*(e_{l+1}),\dots,\phi_*(e_n)\}$ is a basis of the Lie algebra of $\phi(\mathbb G)$. By the fact that $\phi_*$ preserves the stratification (see \autoref{rem:CarnotHom}), we get 
		$$
		\deg \phi_*(e_i)=\deg e_i
		$$
		for each $l+1\leq i\leq n$. Then by \cite[Theorem 4.4, (iii)]{LDNG19} and the previous equation we get 
		$$
		\dim_H \phi(\mathbb G)=\sum_{i=l+1}^n \deg \phi_*(e_i) = \sum_{i=l+1}^n \deg e_i < \sum_{i=1}^n \deg e_i =\dim_H\mathbb G.
		$$
		where we used in the strict inequality that $l>0$ being  $\phi$ not injective. 
	\end{proof}
	
		From \autoref{rem:Purely2} we have this consequence to \autoref{cor:T1.5}.
	\begin{corollary}\label{thm:T1.5}
		There exist a Carnot group $\mathbb G$ and a $C^{\infty}$ non-characteristic hypersurface $S\subseteq \mathbb G$ that is not Pauls Carnot rectifiable according to \autoref{def:LipCarRect}.
	\end{corollary}
	\begin{remark}
		The examples of \autoref{thm:T1.5} are actually $C^1_{\rm H}$-hypersurfaces because they are smooth and non-characteristic, see \autoref{rem:Characteristic}. Thus they rectifiable in the sense of Franchi, Serapioni and Serra Cassano but we proved they are not in the sense of \cite[Definition 4.1]{Pau04}. Indeed, the definition of Pauls Carnot rectifiability is a generalization of \cite[Definition 4.1]{Pau04}, see \autoref{rem:CountLip}.
	\end{remark}
	\begin{remark}\label{rem:TangentRect}
		We notice that every tangent group to $S$ as in the proof of \autoref{cor:T1.5} is a Carnot group. So $S$ is an example of a smooth non-characteristic hypersurface in a Carnot group that cannot be Lipschitz parametrizable by countably many subsets of its tangents. 
	\end{remark}
	
	We state here as a theorem something we already proved in \autoref{thm:T1}.
	\begin{theorem}\label{thm:T2}
	There exists a proper locally doubling metric measure space $(X,d,\mathcal{H}^k)$, where $k:=\mbox{dim}_H X$, that satisfies the following two properties:
	\begin{enumerate}
		\item For each $x\in X$, there exists (up to isometry) only one element in $\mbox{Tan}(X,d,x)$ and it is a Carnot group;
		\item For each $U\subseteq X$ with $\mathcal{H}^k(U)>0$ there exists an uncountable family $\{x_i\}_{i\in I}\subseteq U$ of points such that the tangent spaces at these points are pairwise non-bi-Lipschitz equivalent.
	\end{enumerate}
	\end{theorem}
\begin{proof}
	The example and the proof is exactly the same as in the proof of \autoref{thm:T1}.
\end{proof}
\begin{remark}
Another example (a sub-Riemannian manifold) that satisfies \autoref{thm:T2} was presented in \cite[Proposition 13]{LDY19}.
\end{remark}
\begin{remark}($S$ has a structure of sub-Riemannian manifold)\label{rem:SubRiemannian}
	For an introduction to sub-Riemannian manifolds, also in the Finsler case, one can see \cite[Chapter 2]{LD17b}.
	We show that the example in \autoref{rem:Ex2} has also the structure of an equiregular sub-Riemannian manifold.  Indeed, the set $S$ is a smooth hypersurface of $\mathbb{R}^8$ and we can consider the distribution
	$$
	\mathcal{D}_x:=\mbox{span}\{X_1(x),X_2(x)-x_2^2X_0(x),X_3(x)\}\subseteq T_xS,
	$$
	where $X_i(x)$ are defined in \autoref{ExpressionsOfVector} and $T_xS$ is the Euclidean tangent of $S$. We equip the distribution $\mathcal{D}$ with some smooth scalar product $\bold g$. By the computations made in \autoref{prop:R1} and by considering \eqref{ExpressionsOfVector}, we get 
	\begin{equation}\label{Deltan}
	\begin{aligned}
	\mathcal{D}^2_x&:=\mathcal{D}_x +[\mathcal{D},\mathcal{D}]_x=\\
	&=\mbox{span}\{X_1(x),X_2(x)-x_2^2X_0(x),X_3(x),X_4(x),X_5(x),X_6(x)\}, \\ 
	\mathcal{D}^3_x&:= \mathcal{D}^2_x+[\mathcal{D},\mathcal{D}^2](x)=\\
	&=\mbox{span}\{X_1(x),X_2(x)-x_2^2X_0(x),X_3(x),X_4(x),X_5(x),X_6(x),X_7(x)\}=\\
	&=T_xS.
	\end{aligned}
	\end{equation}
	Then we get that, for all $x\in S$, $\mbox{dim}\mathcal{D}_x= 3$, $\mbox{dim}\mathcal{D}^2_x = 6$ and $\mbox{dim}\mathcal{D}^3_x = 7$. Thus we define the sub-Riemannian distance $d_{sR}$ associated to the sub-Riemannian structure $(S,\mathcal{D},\bold g)$. We get from the results in \cite{Bel96} (see also \cite[Theorem 2.5]{Jean14}, \cite[page 25]{Jean14}) and the explicit expression \eqref{Deltan}, that the tangent space at $x\in S$ to $(S,d_{sR})$ is isometric to the Carnot group $\mathbb{G}_{-x_2^2}$ with Lie algebra $\Gg_{-x_2^2}$, defined in \autoref{def:Gmu}, equipped with the Carnot-Carathéodory distance induced by the left-invariant scalar product that, on the first stratum of the Lie algebra, coincide with $\bold g_x$. Also from \cite[Theorem 2]{Mit85} (see also \cite[Theorem 3.1]{GJ16}) we get that the Hausdorff dimension of $(S,d_{sR})$ is 12. 
	
	Then a slightly change of the proof of \autoref{thm:T1} - namely at the end of the proof we need to use the results about the Hausdorff dimension of smooth submanifolds in a sub-Riemannian manifold in  \cite[0.6.B]{Grom96} (see also \cite[Theorem 5.3]{GJ16}) to obtain $\mathcal{H}^{12}(S\cap\{x_2=\lambda\})=0$ - gives that $(S,d_{sR})$ is purely bi-Lipschitz homogeneous unrectifiable. Also, reasoning exactly as in the proof of \autoref{thm:T1}, we have that $(S,d_{sR},\mathcal{H}^{12})$ satisfies \autoref{thm:T2}.
\end{remark}
\section{Pauls' rectifiability of $C^{\infty}$-hypersurfaces in Heisenberg \\ groups}\label{sec:MainResults2}
	In this section we prove that $C^{\infty}$-hypersurfaces in the $n$-th Heisenberg group, $\mathbb H^n$ with $n\geq 2$, are rectifiable according to \cite[Definition 3]{CP04}, see also \autoref{thm:FinalTheorem} and \autoref{rem:BiLipPauls}. We start with a distance comparison lemma which more generally holds for Carnot groups of step 2, see \autoref{prop:OnB}. Then in Section \ref{subsec:Equivalence} we show the equivalence of intrinsic distance and induced distance for intrinsic Lipschitz graphs in Heiseneberg groups, which has been suggested to us by F\"assler and Orponen, adapting an argument of \cite{CMPSC14}. After that, in Section \ref{1.3}, we prove that $C^{\infty}$ non-characteristic hypersurfaces in $\mathbb H^n$, with $n\geq 2$, carry a sub-Riemannian structure, see \autoref{prop:MainProp2} and \cite[Theorem 1.1]{TY04}. Therefore, we show that the sub-Riemannian distance is locally equivalent to the induced distance, see \autoref{prop:DistanceAreEquivalent}. By means of \cite[Theorem 1]{LDY19} we are able to conclude the result: see \autoref{thm:FinalTheorem}.
	
	\subsection{Carnot Groups of step 2}\label{sub:CarnotStep2}
	In this subsection we recall the geometry of step 2 groups in exponential coordinates. We stress here that sometimes we use Einstein notation: we do not use the symbol $\Sigma$ and we remind that, in this case, we are tacitly taking the sum over the repeated indexes. Every Carnot group of step 2 arises as follows.

	Let $(B^1_{jl}), \dots, (B^n_{jl})$ be $n$ linearly independent skew-symmetric $m\times m$ matrices  with $j,l=1,\dots,m$. Consider the Carnot group $\left(\setR^m\times\setR^n,\cdot,\delta_{\lambda}\right)$ where the operation is 
	\begin{equation}\label{eqn:Product}
	\begin{aligned}
	(x_1,\dots,x_m,y_1,\dots,y_n)\cdot \left(\tilde x_1,\dots,\tilde x_m, \tilde y_1,\dots,\tilde y_n\right) := \\
	\left(x_1+\tilde x_1,\dots,x_m+\tilde x_m,y_1+\tilde y_1+\frac{1}{2}B^1_{jl}\tilde x_j x_l,\dots, y_n+\tilde y_n+\frac{1}{2}B^n_{jl}\tilde x_j x_l\right),
	\end{aligned}
	\end{equation} 
	and the dilations are
	\begin{equation}
	\delta_{\lambda}(x_1,\dots,x_m,y_1,\dots,y_n):=(\lambda x_1,\dots,\lambda x_m,\lambda^2 y_1,\dots,\lambda^2 y_n),
	\end{equation}
	for every $\lambda >0$.

	In the Carnot groups defined above we call $X_j$, with $j=1,\dots,m$, the left-invariant vector fields that agree with $\partial_{x_j}$ at the origin. We call $Y_k$, with $k=1,\dots,n$, the left-invariant vector fields that agree with $\partial_{y_k}$ at the origin. It holds 
	\begin{equation}\label{rem:LeftInvariantVectorFields}
	\begin{aligned}
	X_j&=\partial_{x_j}+\frac{1}{2}B^k_{jl}x_l\partial_{y_k}, \\
	Y_k&=\partial_{y_k}.
	\end{aligned}
	\end{equation}
	We shall consider the two following homogeneous subgroups
	\begin{equation}\label{def:LandW}
	\mathbb L:=\{(x_1,0,\dots,0)\}, \qquad \mathbb W:=\{(0,x_2,\dots,x_m,y_1,\dots,y_n)\}.
	\end{equation}
	In what follows $\tilde{\phi}:\mathbb W\to \mathbb L$ will be an intrinsic $L$-Lipschitz function 
	(\autoref{def:IntrinsicLipschitz})  and $\phi: \setR^{m+n-1}\to \setR$ is defined according to 
	\begin{equation}\label{eqn:Phi}
	\tilde{\phi}(0,x_2,\dots,x_m,y_1,\dots,y_n)=(\phi(x_2,\dots,x_m,y_1,\dots,y_n),0,\dots,0).
	\end{equation}
	The following vector fields on $\mathbb W$ are strictly related to the intrinsic gradient of a function, see \cite[Section 5]{DiDonato18}.
	\begin{definition}\label{def:VectorOnW}
		Given $\tilde{\phi}$ and $\phi$ as in $\eqref{eqn:Phi}$ we define, for $j=2,\dots,m$, the vector fields at $\mathbb W$ on $\bar x:= (0,x_2,\dots,x_m,y_1,\dots,y_n)$ as
		\begin{equation}\label{eqn:DPhi}
		D_j^{\phi}|_{\bar x} := X_j|_{\bar x}+\phi(x_2,\dots,x_m,y_1,\dots,y_n) B^k_{j1}Y_k|_{\bar x}.
		\end{equation}
	\end{definition}
	\begin{definition}\label{def:LengthOfACurve}
		We will say that an absolutely continuous curve $
		\tilde{\gamma}:I\to\mathbb W$ 
		is \emph{horizontal for the family of vector fields} $\{D_j^{\phi}\}_{j=2,\dots,m}$, if there exist $(a_2(t),\dots,a_m(t))\in L^1(I;\mathbb R^{m-1})$ such that 
		\begin{equation}\label{eqn:Controls}
		\tilde{\gamma}'(t)=a_j(t)D_j^{\phi}|_{\tilde{\gamma}(t)}, \qquad \mbox{for a.e.}\quad t\in I.
		\end{equation}
		Then the $\phi$-\emph{length} of $\tilde{\gamma}$ is defined as
		\begin{equation}\label{eqn:LengthPhi}
		\ell_{\phi}(\tilde{\gamma}):=\int_I\sqrt{a_2(s)^2+\dots+a_m(s)^2}\de s.
		\end{equation}
	\end{definition}
	\begin{remark}
		Notice that if 
		\begin{equation}\label{eqn:TildeGamma}
		\tilde{\gamma}(t):=(0,x_2(t),\dots,x_m(t),y_1(t),\dots y_n(t)),
		\end{equation}
		then 
		\begin{equation}\label{eqn:LengthInCoordinates}
		\ell_{\phi}(\tilde{\gamma})=\int_I\sqrt{x_2'(s)^2+\dots+x_m'(s)^2}\de s.
		\end{equation}
	\end{remark}
	\begin{remark}\label{rem:HeisAsStep2}
	Using the notation in Section \ref{sub:CarnotStep2}, the Heisenberg group $\mathbb H^{\bar n}$ is obtained when $m=2\bar{n}$, $n=1$ and $B^1_{ij}=1$ if and only if $i=j+\bar{n}$, otherwise it is zero.
	\end{remark}
	
	\subsection{Length comparison for Carnot groups of step 2}
	In this subsection we will show that for Carnot groups of step 2, the length measured with the left-invariant homogeneous distance $d$ in the group $\mathbb G$ of the curve $\tilde{\gamma}\cdot \tilde{\phi}(\tilde{\gamma})$ is controlled from above by $\ell_{\phi}(\tilde{\gamma})$. 
\begin{remark}\label{rem:Length}
	For the general theory of sub-Riemannian manifolds, including the Finsler case, one can check \cite[Chapter 2]{LD17b}. We recall that in $\mathbb G$ we have two interpretations for the length of an absolutely continuous curve. Indeed, as in any Carnot-Carathéodory space, if the distance $d$ on $\mathbb G$ is induced by a norm $\|\cdot\|$ on the horizontal bundle $\mathbb V_1$ of $\mathbb G$, the length of a continuous curve $\gamma:I\to \mathbb G$ equals the following values
	\begin{equation}\label{def:Length}
	\mbox{length}(\gamma):=\sup\left\{\sum_{i=1}^n d(\gamma(s_{i-1}),\gamma(s_i))\right\}=\int_I \bold \|\gamma'(t)\|\de t,
	\end{equation}
	where the sup is over the partitions $\sqcup_{i=0}^n [s_{i},s_{i+1}]$ of $I$.

\end{remark}
The proof of the forthcoming proposition was pointed out to us by F\"assler and Orponen in the Heisenber group and it is substantially contained in \cite[Proposition 3.8.]{CMPSC14}. We present here a general proof for step 2 groups. 
\begin{proposition}\label{prop:OnB}
	Let $\mathbb G$ be a step 2 Carnot group with the choice of coordinates as in Section \ref{sub:CarnotStep2} and $\mathbb W$ and $\mathbb L$ as in \eqref{def:LandW}. Let $\tilde{\phi}:\mathbb W\to\mathbb L$ be intrinsic $L$-Lipschitz. Set $\phi,D_j^{\phi},\ell_{\phi}$ as in \eqref{eqn:Phi}, \eqref{eqn:DPhi}, and \eqref{eqn:LengthPhi}, respectively. If $\tilde{\gamma}:I\to\mathbb W$ is horizontal with respect to $\{D_j^{\phi}\}_{j=2,\dots,m}$, then 
	$$
	{\rm length}(\tilde{\gamma}\cdot\tilde{\phi}(\tilde{\gamma})) \leq C\cdot \ell_{\phi}(\tilde{\gamma}),
	$$
	where $C=C(\mathbb G,L)$.
	
	Moreover, if the norm of the controls $a_j(t)$ of $\tilde{\gamma}$ as in \autoref{def:LengthOfACurve} are bounded by $K$, the projection on the first component of the curve $s\mapsto \tilde{\phi}(\tilde{\gamma}(s))$ is $L'$-Lipschitz, with $L'=L'(L,K,\mathbb G)$.
\end{proposition}
\begin{proof}
	Set $\gamma:I\to \setR^{m+n-1}$ be  the curve 
	$$
	\gamma(t):=(x_2(t),\dots,x_m(t),y_1(t),\dots,y_n(t)),
	$$
	where we use the notation \eqref{eqn:TildeGamma}.
	By the fact that $\tilde{\gamma}$ is horizontal with respect to $\{D^{\phi}_j\}_{j=2,\dots,m}$ we get by easy computations that for each $k=1,\dots,n$ 
	\begin{equation}\label{eqn:Horizontal}
	y_k'(t)=x_j'(t)\left(\frac{1}{2}B_{jl}^kx_l(t)+\phi(\gamma(t))B_{j1}^k\right), \qquad \mbox{for a.e.} \quad t\in I,
	\end{equation}
	where we sum over $j$ and $l$ from $2$ to $m$. Now we consider the curve $\tilde{\gamma}$ between two intermediary times $t<t_1$ and we claim that 
	\begin{equation}\label{eqn:EstSum}
	\sum_{i=2}^m |x_i(t_1)-x_i(t)| \leq C_1\ell_{\phi}\left(\tilde{\gamma}|_{[t,t_1]}\right),
	\end{equation}
	where the constant $C_1=C_1(m)$. Indeed, this is a consequence of the fundamental theorem of calculus, Cauchy-Schwarz and  \eqref{eqn:LengthInCoordinates}.
	
	Set $
	\Phi(\tilde{\gamma}(\cdot)):=\tilde{\gamma}(\cdot)\tilde{\phi}(\tilde{\gamma}(\cdot))$.
	By the definition of length it suffices to show that for all $[t,t_1]\subseteq I$ there exists a constant $C=C(L,\mathbb G)$ such that
	\begin{equation}\label{eqn:Final}
	d(\Phi(\tilde{\gamma}(t)),\Phi(\tilde{\gamma}(t_1)) \leq C\cdot \ell_{\phi}(\tilde{\gamma}|_{[t,t_1]}).
	\end{equation}
	By the fact that $\phi$ is intrinsic Lipschitz and \cite[Proposition 2.3.4.]{FMS14} one has that, setting $\|\cdot\|$ the homogeneous norm on $\mathbb G$ associated to $d$ (see Section \ref{sub:Carnot}), there exists a constant $C_0=C_0(L)$ such that 
	\begin{equation}\label{eqn:FinalMinusOne}
	d(\Phi(\tilde{\gamma}(t)),\Phi(\tilde{\gamma}(t_1)) \leq C_0\|\pi_{\mathbb W}\left(\Phi(\tilde{\gamma}(t))^{-1}\Phi(\tilde{\gamma}(t_1))\right)\|.
	\end{equation}
	Then we leave to the reader to verify the algebraic equality
	\begin{equation}\label{eqn:CompOnProject}
	\pi_{\mathbb W}\left(\Phi(\tilde{\gamma}(t))^{-1}\Phi(\tilde{\gamma}(t_1))\right) = \tilde{\phi}(\tilde{\gamma}(t))^{-1}\tilde{\gamma}(t)^{-1}\tilde{\gamma}(t_1)\tilde{\phi}(\tilde{\gamma}(t)).
	\end{equation}
	By exploiting the formula for the group law, it holds
	\begin{equation}\label{eqn:615}
	\begin{aligned}
	&\tilde{\phi}(\tilde{\gamma}(t))^{-1}\tilde{\gamma}(t)^{-1}\tilde{\gamma}(t_1)\tilde{\phi}(\tilde{\gamma}(t)) = \\
	&=(0,x_2(t_1)-x_2(t),\dots,x_m(t_1)-x_m(t),\sigma_1(t_1,t),\dots,\sigma_n(t_1,t)),
	\end{aligned}
	\end{equation}
	where for each $k=1,\dots,n$ we have
	\begin{equation}
	\sigma_k(t_1,t):=y_k(t_1)-y_k(t)+B_{1j}^k\phi(\gamma(t))(x_j(t_1)-x_j(t))-\frac{1}{2}B_{jl}^kx_j(t_1)x_l(t),
	\end{equation}
	where the sums on indexes $j$ and $l$ run from $2$ to $m$. 
	
	Then by \eqref{eqn:615} and the fact that $\|\cdot\|$ is equivalent to any other homogeneous norm on $\mathbb G$, we have that 
	\begin{equation}\label{eqn:Projection}
	\|\tilde{\phi}(\tilde{\gamma}(t))^{-1}\tilde{\gamma}(t)^{-1}\tilde{\gamma}(t_1)\tilde{\phi}(\tilde{\gamma}(t))\| \sim \sum_{i=1}^m |x_i(t_1)-x_i(t)|+\sum_{k=1}^n \sqrt{|\sigma_k(t_1,t)|}.
	\end{equation}
	Using \eqref{eqn:Horizontal} and that $B_{jl}^kx_j(t)x_l(t)=0$ by skew-symmetry of $B^k$ we can rewrite $\sigma_k(t_1,t)$ as follows
	\begin{equation}\label{eqn:ComptOnSigma}
	\begin{aligned}
	\sigma_k(t_1,t)&=\left(\int_t^{t_1} y_k'(\xi)\de \xi\right)+B_{1j}^k\phi(\gamma(t))(x_j(t_1)-x_j(t))-\frac{1}{2}B_{jl}^kx_j(t_1)x_l(t)  \\
	&=\int_t^{t_1}\left(x_j'(\xi)B_{j1}^k\left(\phi(\gamma(\xi))-\phi(\gamma(t))\right)+\frac{1}{2}x_j'(\xi)B_{jl}^k(x_l(\xi)-x_l(t))\right)\de\xi.
	\end{aligned}
	\end{equation}
	Set
	$$
	f(t_1,t):=\sup_{\xi\in[t,t_1]}|\phi(\gamma(\xi)-\phi(\gamma(t)))|.
	$$
	It follows from \eqref{eqn:ComptOnSigma}, \eqref{eqn:EstSum} and Cauchy-Schwarz inequality that 
	\begin{equation}\label{eqn:Sigma}
	|\sigma_k(t_1,t)|\leq C_2\left(\ell_{\phi}(\tilde{\gamma}|_{[t,t_1]})+f(t_1,t)\right)\ell_{\phi}(\tilde{\gamma}|_{[t,t_1]}) ,
	\end{equation}
	where $C_2=C_2(m, B)$ and $B:=\max{|B_{jl}^k|}$.
	Now for each $\xi\in[t,t_1]$ we get by the fact $\phi$ is intrinsic Lipschitz, \eqref{eqn:CompOnProject}, \eqref{eqn:Projection}, \eqref{eqn:EstSum} and \eqref{eqn:Sigma} with $\xi$ instead of $t_1$, that
	\begin{equation}\label{eqn:ForLip}
	\begin{aligned}
	|\phi(\gamma(\xi))-\phi(\gamma(t))|&\leq L\|\pi_{\mathbb W}\left(\Phi(\tilde{\gamma}(t))^{-1}\Phi(\tilde{\gamma}(\xi))\right)\| \\
	&=\|\tilde{\phi}(\tilde{\gamma}(t))^{-1}\tilde{\gamma}(t)^{-1}\tilde{\gamma}(\xi)\tilde{\phi}(\tilde{\gamma}(t))\| \sim \sum_{i=1}^m |x_i(\xi)-x_i(t)|+\sum_{k=1}^n \sqrt{|\sigma_k(\xi,t)|} \\  &\leq C_3\left(\ell_{\phi}(\tilde{\gamma}|_{[t,\xi]})+\sqrt{f(\xi,t)}\sqrt{\ell_{\phi}(\tilde{\gamma}|_{[t,\xi]})}\right),
	\end{aligned}
	\end{equation}
	where $C_3=C_3(m,B,L)$. Now passing to the supremum as $\xi\in[t,t_1]$ in both sides we get
	$$
	f(t_1,t)\leq C_3\left(\ell_{\phi}(\tilde{\gamma}|_{[t,t_1]})+\sqrt{f(t_1,t)}\sqrt{\ell_{\phi}(\tilde{\gamma}|_{[t,t_1]})}\right),
	$$
	from which there exists $C_4=C_4(m,B,L)$ such that 
	\begin{equation}\label{eqn:InequalityOnF}
	f(t_1,t)\leq C_4 \ell_{\phi}\left(\tilde{\gamma}|_{[t,t_1]}\right).
	\end{equation}
	Finally by chaining \eqref{eqn:FinalMinusOne}, \eqref{eqn:CompOnProject}, \eqref{eqn:Projection}, \eqref{eqn:EstSum}, \eqref{eqn:Sigma}, and \eqref{eqn:InequalityOnF}, we get \eqref{eqn:Final} which was what we wanted. For the second part of the lemma we just chain \eqref{eqn:ForLip} and \eqref{eqn:InequalityOnF} with $\xi$ instead of $t_1$, and use the fact that $\ell_{\phi}(\tilde{\gamma}|_{[t,\xi]})$ is bounded above by $C(K,m)|\xi-t|$ by the definition of $\ell_{\phi}$ in \eqref{eqn:LengthPhi}.
\end{proof}
\begin{remark}
	The second part of \autoref{prop:OnB} recovers also the  statement of \cite[Proposition 3.6]{DiDonato19}. Notice that results that are similar to \autoref{prop:OnB} have been proved by Kozhevnikov. In particular, in \cite[Proposition 4.2.16]{Koz15} it is proved, in the general setting of Carnot groups, a characterization of intrinsic Lipschitz graphs by means of metric properties of integral curves, on $\mathbb W$, of some analog of ours $\{D_j^{\phi}\}$ defined in \eqref{eqn:DPhi}. In particular, the implications $1\Rightarrow 2$ and $1\Rightarrow 3$ in \cite[Proposition 4.2.16]{Koz15} show the statement of \autoref{prop:OnB} but just for horizontal curves $\tilde{\gamma}:I\to\mathbb W$ with constant controls $a_j(t)\equiv a_j$ in \eqref{eqn:Controls}. A similar result was already known from \cite[Theorem 1.2]{CM06}.
\end{remark}
\subsection{Equivalence of intrinsic distance and induced distance on intrinsic Lipschitz graphs in the Heisenberg groups}\label{subsec:Equivalence}
\begin{definition}\label{def:IntrinsicDistance}
	Given $\tilde{\phi}$ and $\phi$ as in \autoref{eqn:Phi} we define the intrinsic length distance on the graph $\Gamma:={\rm graph}(\tilde{\phi})=\{w\tilde{\phi}(w):w\in\mathbb W\}\subseteq \mathbb G$ as follows
	\begin{equation}\label{eqn:DGamma}
	d^{\Gamma}(x,y):=\inf\{\mbox{length}(\gamma)|\gamma:[0,1]\to \Gamma, \gamma(0)=x, \gamma(1)=y, \quad \gamma \quad \mbox{horizontal}\}.
	\end{equation}
\end{definition}
\begin{remark}\label{rem:SubRiemannianDistance} 
	Up to a globally bi-Lipschitz change of norm, we can suppose to work with a left-invariant homogeneous distance $d$ on $\mathbb G$ coming from a scalar product $\bold g$ on the horizontal bundle $\mathbb V_1$. Notice that if $\Gamma$ is a smooth submanifold of $\mathbb G$ and the horizontal bundle $\mathbb V_1$ intersects the tangent bundle of $\Gamma$ in a bracket generating distribution, then the distance $d^{\Gamma}(x,y)$ is exactly the sub-Riemannian distance, let us call it $d_{\rm int}(x,y)$, associated to the sub-Riemannian structure $\left(\Gamma,\mathbb V_1\cap T\Gamma,\bold g|_{(\mathbb V_1\cap T\Gamma)\times (\mathbb V_1\cap T\Gamma)}\right)$. 
\end{remark}
We now stress that in the specific case of the Heisenberg groups the induced distance $d$ is bi-Lipschitz equivalent to $d^{\Gamma}$ defined in \eqref{eqn:DGamma}. The proof was suggested to us by F\"assler and Orponen.
\begin{proposition}\label{prop:EquivalentDistances}
	With the same assumptions and notation as in \autoref{prop:OnB}, if $\mathbb G=\mathbb H^n$ with $n\geq 2$, then 

	$$
	d(x,y) \sim d^{\Gamma}(x,y), \qquad \forall x,y \in \Gamma,
	$$
	where $d^{\Gamma}$ is defined in \eqref{eqn:DGamma} and $d$ is the induced distance, restriction of the one in $\mathbb H^n$.
\end{proposition}
\begin{proof} 
	First of all notice that, using the notation in \eqref{eqn:DPhi} and taking into account \autoref{rem:HeisAsStep2}, if $\mathbb G=\mathbb H^n$, then $D^{\phi}_j|_{\bar x}=X_j|_{\bar x}$ for all $j=2,\dots,2n$ and $j\neq n+1$, while $D^{\phi}_{n+1}|_{\bar x}=X_{n+1}|_{\bar x}+\phi(x_2,\dots,x_{2n},y_1)Y_1|_{\bar x}$. We also have $[X_j,X_{n+j}]=Y_1$ for every $j=1,\dots,n$ and all the other commutators are zero.  By definition of $d^{\Gamma}$ \eqref{eqn:DGamma}, exploiting the definition of length \eqref{def:Length} and the triangular inequality, we get
	$$
	d(x,y)\leq d^{\Gamma}(x,y), \qquad \forall x,y \in \Gamma.
	$$ 
	Now we want to prove the opposite inequality up to a constant. First of all, by a left translation, we can assume $x=0$ and $y=w\tilde{\phi}(w)$ for $w\in \mathbb W$. It holds  that 
	\begin{equation}\label{eqn:Step}
	d(x,y)=d(0,y)=\|w\tilde{\phi}(w)\|_{d}\geq C_0\|w\|_{d},
	\end{equation}
	where $\|\cdot\|_{d}$ is the homogeneus norm associated to $d$ and $C_0=C_0(\mathbb W,\mathbb H)$, see \cite[Theorem 2.2.2.]{FMS14}. From now on, in this proof, we will set $\|\cdot\|_d:=\|\cdot\|$.
	
	We claim that we can conclude if we show that for each $w\in \mathbb W$ there exists $\tilde{\gamma}\subseteq \mathbb W$, connecting $0$ to $w$, horizontal for $\{D_j^{\phi}\}_{j=2,\dots,2n+1}$, such that 
	\begin{equation}\label{eqn:Step2}
	{\ell}_{\phi}(\tilde{\gamma})\leq C_1\| w\|,
	\end{equation}
	for some constant $C_1$. Indeed, if \eqref{eqn:Step2} holds, then from the first part of \autoref{prop:OnB} and \eqref{eqn:Step} we get that, setting $\Phi(\tilde{\gamma}):=\tilde{\gamma}\tilde{\phi}(\tilde{\gamma})$,
	$$
	{\rm length}(\Phi(\tilde{\gamma})) \leq C_2d(x,y),
	$$
	where $C_2$ is a constant, and $\Phi(\tilde{\gamma})$ is a curve contained in $\Gamma$ connecting $x=0$ to $y=w\tilde{\phi}(w)$. Being the length of $\Phi(\tilde{\gamma})$ finite, we get that it is a horizontal curve \cite[Theorem 2.4.5.]{LD17b} and then we get 
	$$
	d^{\Gamma}(x,y)\leq C_2d(x,y),
	$$
	that finishes the proof. 
	
	Now we show the existence of $\tilde{\gamma}$, with the required properties, such that \eqref{eqn:Step2} holds. We  concatenate two curves $\tilde{\gamma}_1$ and $\tilde{\gamma}_2$, horizontal for $\{D_j^{\phi}\}_{j=2,\dots,2n+1}$, to reach $w:=(0,x_2,\dots,x_{2n},y_1)$ from $0$. Due to the fact that $\phi$ is continuous, because of \autoref{rem:IntrinsicIsContinuous}, Peano's theorem \cite[Theorem 1.1]{Hal80} ensures that there exists a local solution to  the continuous ODE
	\begin{equation}\label{eqn:System}
	\begin{cases}
		\tau'(s)&=\phi(0,\dots,0,s,0,\dots,0,\tau(s)), \\
		\tau(0)&=0,
	\end{cases}
	\end{equation}
	where $s$ in the $(n+1)$-th coordinate. Set
	$$
	\tilde{\gamma}_1(s):=(0,\dots,0,s,0,\dots,0,\tau(s)),
	$$ 
	so that by \eqref{eqn:System} it holds 
	\begin{equation}\label{eqn:BoundedControl}
	\tilde{\gamma}_1'(s)=D_{n+1}^{\phi}(\tilde{\gamma}_1(s)).
	\end{equation}
	We show that $\tau(s)$ is defined globally on $\mathbb R$, arguing similarly as in \cite[(4.1) and after]{CFO19}. Indeed, whenever $\tau(s)$ exists,
	\begin{equation}\label{eqn:EstimateOnT}
	\tau(s)=\int_0^s \tau'(\xi)\de\xi=\int_0^s\phi(\tilde{\gamma}_1(\xi))\de\xi,
	\end{equation}
	and by \eqref{eqn:BoundedControl} and the second part of \autoref{prop:OnB} we have that $s\mapsto \phi(\tilde{\gamma}_1(s))$ is $L'$-Lipschitz, with $L'=L'(L)$. By \eqref{eqn:EstimateOnT} we have 
	\begin{equation}\label{eqn:Comp}
	|\tau(s)|\leq L's^2.
	\end{equation}
	Thus, as any solution to \eqref{eqn:System} escapes every compact set \cite[Theorem 2.1]{Hal80}, we get from \eqref{eqn:Comp} that $\tau(s)$ is globally defined. Then $\tau(s)$ is defined up to $s=x_{n+1}$ and by the previous argument 
	\begin{equation}\label{eqn:Tau}
	|\tau(x_{n+1})|\leq L'x_{n+1}^2. 
	\end{equation}
	We notice that we can identify the  arbitrary point $(x_2,\dots,x_n,x_{n+2},\dots,x_{2n},y_1)$ with a point in $\mathbb H^{n-1}$. Thus we can connect the point $(0,\dots,0,\tau(x_{n+1}))$, where we just removed the first and the $(n+1)$-th coordinate from $\tilde{\gamma}_1(x_{n+1})$, to the point $(x_2,\dots,x_n,x_{n+2},\dots,x_{2n},y_1)$, by using a horizontal geodesic in $\mathbb H^{n-1}$ with respect to the Carnot-Carathéodory  distance $d_{\bold g}$ induced, on $\mathbb H^{n-1}$, by the scalar product $\bold g$ that makes $X_2,\dots,X_n,X_{n+2},\dots,X_{2n}$ orthonormal. We set $\tilde{\gamma}_2:I\to\mathbb H^n$ to be this horizontal geodesic in $\mathbb H^{n-1}$, read in $\mathbb{H}^n$. We notice that it is horizontal with respect to the family $\{D_j^{\phi}\}_{j=2,\dots,n,n+2,\dots,2n}$, because $D_j^{\phi}=X_j$ for $j=2,\dots,2n$ and $j\neq n+1$. Then we have 
	\begin{equation}\label{eqn:FIN1}
	\begin{aligned}
	{\ell}_{\phi}(\tilde{\gamma}_2)&=d_{\bold g}((0,\dots,0,\tau(x_{n+1})),(x_2,\dots,x_n,x_{n+2},\dots,x_{2n},y_1))\\
	&\leq C_3\left(|y_1-\tau(x_{n+1})|^{1/2}+\sum_{i=2,i\neq n+1}^{2n}|x_i|\right)\\
	&\leq C_4\left(|y_1|^{1/2}+\sum_{i=2}^{2n}|x_i|\right)\leq C_5\|w\|,
	\end{aligned}
	\end{equation}
	where the first equality follows by the definition of ${\ell}_{\phi}$ \eqref{eqn:LengthPhi} and the fact that $\tilde{\gamma}_2$, restrict to $\mathbb H^{n-1}$, is a $d_{\bold g}$-geodesic; the second is true because any two homogeneous norms are equivalent, the third one is true because of \eqref{eqn:Tau}, and the last one again by the fact that any two homogeneous norms are equivalent. Now we have 
	\begin{equation}\label{eqn:FIN2}
	{\ell}_{\phi}\left(\tilde{\gamma}_1|_{[0,x_{n+1}]}\right)=|x_{n+1}|\leq C_6\|w\|,
	\end{equation}
	where the first equality is true by the definition of ${\ell}_{\phi}$ and \eqref{eqn:BoundedControl} and second is true again because of the equivalence of homogeneous norms. Now if we set $\tilde{\gamma}:=\tilde{\gamma}_1|_{[0,x_{n+1}]}\star\tilde{\gamma}_2$ the concatenation of the two curves, we get that $\tilde{\gamma}$ is horizontal and summing \eqref{eqn:FIN1} and \eqref{eqn:FIN2} we get \eqref{eqn:Step2} with $C_1:=C_5+C_6$, which was what was left to prove.
\end{proof}

\subsection{Sub-Riemannian structure of a $C^{\infty}$ non-characteristic hypersurface in the Heisenberg groups}\label{1.3}

\subsubsection{The restriction of the horizontal bundle is bracket generating}
Now we are going to prove that for non-characteristic $C^{2}$-hypersurfaces (see \autoref{def:CharacteristicSurface}) in $\setH^n$, with $n\geq 2$, the intersection between the horizontal bundle and the tangent bundle is bracket generating.
This result was already known and it is a consequence of a more general one \cite[Theorem 1.1.]{TY04}. Nevertheless we give here a simple proof by making explicit computations. 
\begin{proposition}\label{prop:MainProp2}
	Consider in $\setH^n$, with $n\geq 2$, a $C^{2}$-hypersurface $S$. If $S$ has no characteristic points, then the  bundle
	\begin{equation}\label{eqn:D}
	x\mapsto \mathcal{D}_x := \mathbb V_1(x) \cap T_xS,
	\end{equation}
	gives a step-2 bracket generating distribution on the hypersurface $S$.
\end{proposition}

\begin{proof}
	We need to prove that 
	$$
	\forall x\in S, \qquad \mathcal{D}_x + [\mathcal{D},\mathcal{D}]_x = T_xS.
	$$
	Let us give the proof first for $n=2$. We work locally around $x\in S$ so that we can assume that there exists $f\in C^{2}(\setH^n)$ such that
	$$
	S= \{x\in \setH^n : f(x)=0 \}.
	$$
	We define locally the vector fields 
	\begin{align*}
	Y_1:= -(X_2f)X_1+(X_1f)X_2-(X_4f)X_3+(X_3f)X_4, \\
	Y_2:= -(X_3f)X_1+(X_4f)X_2+(X_1f)X_3-(X_2f)X_4, \\
	Y_3:= -(X_4f)X_1-(X_3f)X_2+(X_2f)X_3+(X_1f)X_4.
	\end{align*}
	
	We have that for each $x\in S$, the linear space $\mathcal{D}_x$ is a three-dimensional subspace, because $x$ is a non-characteristic point, and then it is easy to see that 
	\begin{equation}\label{WantTo}
	\mathcal{D}_x = \mbox{span}\{Y_1|_x,Y_2|_x,Y_3|_x\}, \qquad \forall x\in S.
	\end{equation}
	Now by doing the computations exploiting the definition of $Y_i$, and using that $[X_1,X_3]=[X_2,X_4]=X_5$, we can show that 
	\begin{align*}
	[Y_1,Y_2] = \star_1 - 2(X_1fX_2f+X_3fX_4f)X_5, \\
	[Y_1,Y_3] = \star_2 + ((X_1f)^2+(X_3f)^2- (X_2f)^2-(X_4f)^2)X_5, \\
	[Y_2,Y_3] = \star_3 + 2(X_1fX_4f-X_2fX_3f)X_5,
	\end{align*} 
	where $\star_1, \star_2, \star_3$ are some combinations of $X_1, X_2, X_3, X_4$ with function coefficients.
	
	It is easy to check that it is not possible to have, at some point $x\in S$,
	\begin{align*}
	(X_1f)(x)(X_2f)(x)+(X_3f)(x)(X_4f)(x)=0, \\
	(X_1f)^2(x)+(X_3f)^2(x)- (X_2f)^2(x)-(X_4f)^2(x) = 0, \\
	(X_1f)(x)(X_4f)(x)-(X_2f)(x)(X_3f)(x)=0,
	\end{align*}
	because otherwise $(X_1f)(x)=(X_2f)(x)=(X_3f)(x)=(X_4f)(x)=0$, which is impossible because there are no characteristic points. Then at least one among $[Y_1,Y_2]|_x, [Y_1,Y_3]|_x, [Y_2,Y_3]|_x$ has a component along $X_5|_x$. Then, as in each point $X_1|_x, X_2|_x, X_3|_x, X_4|_x$ and $X_5|_x$ are linearly independent, this means that there exists at least one among $[Y_1,Y_2]|_x, [Y_1,Y_3]|_x, [Y_2,Y_3]|_x$ which is not in $\mathcal{D}_x$. Then as $[\mathcal{D},\mathcal{D}]_x \subseteq T_xS$, and it holds that there exists an element in $[\mathcal{D},\mathcal{D}]_x$ which is not in $\mathcal{D}_x$, we get the conclusion.
	If $n>2$ we can argue exactly in the same way. Indeed, there exists $i$ with $1\leq i\leq 2n$ such that $X_if(x)\neq 0$ and one runs the same argument  substituting $X_1,X_2,X_3,X_4$ with $X_i,X_j,X_{i+n},X_{j+n}$ with $j\neq i$.
\end{proof}

\subsubsection{Local equivalence of the sub-Riemannian distance and the induced distance}\label{1.4}
Let $S$ be a smooth non-characteristic hypersurface in the Heisenberg group $\mathbb H^n$, $n \geq 2$.
From \autoref{prop:MainProp2} we have a bracket generating distribution $\mathcal{D}$ in the Euclidean tangent bundle $TS$ of $S$. Hence $S$ has the structure of sub-Riemannian manifold: we fix a scalar product $\bold g$ on $\mathbb V_1$, the horizontal bundle of $\mathbb H^n$, which induces a scalar product on $\mathcal{D}$. This scalar product defines a sub-Riemannian distance on $S$ by taking the infimum of the length - measured with the norm $\|\cdot\|_{\bold g}$ associated to $\bold g$ - of all the horizontal - according to $\mathcal{D}$ - curves in $S$. 
We will call this distance $d_{\rm int}$, the \emph{intrinsic distance} on $S$. 
We can also equip $S$ with the restriction of the distance of $\mathbb H^n$, which we will induced distance and with a little abuse of notation we denote it by $d$.

\begin{proposition}\label{prop:DistanceAreEquivalent}
	Let $(\mathbb H^n,d)$, with $n\geq 2$, be the Heisenberg group equipped with the sub-Riemannian distance coming from a scalar product $\bold g$ on the horizontal distribution. Let $S$ be a $C^{\infty}$ non-characteristic hypersurface in $\mathbb H^n$. For each $p\in S$ there exists an open neighbourhood $U_p$ of $p$ such that 
	$$
	d(x,y) \sim d_{\rm int}(x,y) \qquad \forall x,y\in U_p,
	$$
	where $d_{\rm int}$ has been introduced at the beginning of Section \ref{1.4}.
\end{proposition}
\begin{proof}
	By \autoref{prop:C1H=iLip} we get that locally around $p\in S$, the hypersurface $S$ is the graph $\Gamma$ of a globally defined intrinsic Lipschitz function on the tangent group $\mathbb W=T_x^{\rm I}S$. By changing coordinates, we can assume $\mathbb W$ as in \eqref{def:LandW}. Then by \autoref{prop:EquivalentDistances} we get that $d^{\Gamma}\sim d$ and from \autoref{rem:SubRiemannianDistance} we get that, locally on $S$, $d_{\rm int}=d^{\Gamma}$, so that we get the result. 
\end{proof}
\subsubsection{Tangents of $C^{\infty}$ non-characteristic hypersurfaces}
Now we know that a $C^{\infty}$ non-characteristic hypersurface in the Heisenberg groups $\mathbb H^n$, $n\geq 2$, is a sub-Riemannian manifold. With the aim of using the rectifiability result from \cite[Theorem 1]{LDY19}, we calculate the possible tangents of $S$. We begin with a preparatory lemma.
\begin{lemma}\label{lem:IsoSub}
	Every vertical subgroup of codimension one in $\mathbb H^n$, $n\geq 2$, is isomorphic to $\mathbb H^{n-1}\times \mathbb R$, which is a Carnot subgroup.
\end{lemma}
\begin{proof}
	Let us take $\Gw$ a codimension-one subspace of the first stratum $\Gg_1$ of the Lie algebra of $\mathbb H^n$, which we call $\Gh^n$, and let us consider a decomposition $\Gg_1=\mathbb R z\oplus \Gw$. If we consider the brackets restricted to $\Gw$, which is odd-dimensional, we have by skew-symmetry that there exists $v\in \Gw$ such that 
	$$
	[v,w]=0, \qquad \forall w\in \Gw. 
	$$
	Now if we consider a decomposition $\Gw=\mathbb Rv\oplus \Gw'$, we claim that $[\cdot,\cdot]$ on $\Gw'$ is skew-symmetric and non-degenerate. Indeed, if it is not so, there exists $v'\in\Gw'$ such that $[v',w']=0$ for every $w'\in \Gw'$. Now we claim that there exist real numbers $\alpha,\beta$ such that $v'':=\alpha v+\beta v'$ satisfies $[v'',h]=0$ for every $h\in \Gg_1$. Indeed, if $h\in \Gw$, this holds just by how $v$ and $v'$ are defined. Then we can choose $\alpha,\beta$ such that $[v'',z]=0$. Then this is a contradiction because $[\cdot,\cdot]$ is non degenerate on $\Gg_1$. 
	
	Then by the fundamental theorem of symplectic forms $\Gw'$ is isomorphic as Lie algebra to $\Gh^{n-1}$, from which the thesis follows. 
\end{proof}
\begin{proposition}\label{prop:EverywhereIsomorphic}
	Let $S$ be a $C^{\infty}$-hypersurface in $\mathbb H^n$, $n\geq 2$, with no characteristic points. Let $\mathcal{D}$ be as in \eqref{eqn:D} and $\bold g$ be a scalar product on the horizontal bundle $\mathbb V_1$ of $\mathbb H^n$. 
	
	Then the triple $\left(S,\mathcal{D},\bold g|_{\mathcal{D}\times \mathcal{D}}\right)$ is an equiregular sub-Riemannian manifold with Hausdorff dimension $2n+1$. At each point $x\in S$ we have that the Gromov-Hausdorff tangent is unique and it is isometric the Carnot group $\mathbb H^{n-1}\times \mathbb R$ endowed with some Carnot-Carathéodory distance.
\end{proposition}
\begin{proof}
	Because of the fact that $S$ is non-characteristic it follows that $\mathcal{D}_x$ has dimension $2n-1$ at each point $x\in S$. Also it is a direct consequence of \autoref{prop:MainProp2} that, for each $x\in S$, the linear space $\mathcal{D}_x+[\mathcal{D},\mathcal{D}]_x$ has dimension $2n$. Then $\left(S,\mathcal{D},\bold g|_{\mathcal{D}\times \mathcal{D}}\right)$ is an equiregular sub-Riemannian manifold with weights $(2n-1,1)$. Then the Hausdorff dimension of $S$ with respect to the associate sub-Riemannian distance $d_{\rm int}$ is $2n+1$ by \cite[Theorem 2]{Mit85} (see also \cite[Theorem 3.1]{GJ16}).
	
	By \cite{Bel96} (see also \cite[Theorem 2.5]{Jean14}, \cite[page 25]{Jean14}) it follows, as we are in the equi-regular case, that the Gromov-Hausdorff tangent at any point $x\in S$ is isometric to the Carnot group 
	$$
	\mathbb V_x:=\mathcal{D}_x\oplus ((\mathcal{D}_x+[\mathcal{D},\mathcal{D}]_x)/\mathcal{D}_x)
	$$
	with the bracket operation inherited by the brackets in the Heisenberg group, endowed with some Carnot-Carathéodory distance. Then $\mathbb V_x$ is isomorphic to a vertical subgroup of $\mathbb H^n$ of codimension $1$ and thus isomorphic to $\mathbb H^{n-1}\times \mathbb R$ by \autoref{lem:IsoSub}.
\end{proof}
\subsubsection{Carnot-rectifiability of $C^{\infty}$-hypersurfaces}
We conclude with the main result of this section.
\begin{theorem}\label{thm:FinalTheorem}
	Let $(\mathbb H^n,d)$, with $n\geq 2$, be the $n$-th Heisenberg group equipped with a left-invariant homogenenous distance $d$. If $S$ is a $C^{\infty}$-hypersurface in $\mathbb{H}^n$, then the metric space $(S,d)$ has Hausdorff dimension $2n+1$ and it is  $(\{\mathbb H^{n-1}\times \mathbb R\},\mathcal{H}^{2n+1})$-rectifiable according to \autoref{def:Grect}.
\end{theorem}
\begin{proof}
	Let us assume first that $S$ has no characteristic points. In this case it directly follows from \cite[Theorem 1]{LDY19} and \autoref{prop:EverywhereIsomorphic} that the metric space $(S,d_{\rm int})$ has Hausdorff dimension $2n+1$ and it is $(\{\mathbb H^{n-1}\times \mathbb R\},\mathcal{H}^{2n+1}_{d_{\rm int}})$-rectifiable according to \autoref{def:Grect}. Then by  \autoref{prop:DistanceAreEquivalent} we obtain that $(S,d)$ is $(\{\mathbb H^{n-1}\times \mathbb R\},\mathcal{H}^{2n+1}_d)$-rectifiable.
	
	In the general case, calling $\Sigma_S$ the set of characteristic points, we know that $\mathcal{H}^{2n+1}\left(\Sigma_S\right)=0$ by \cite[Theorem 1.1]{Bal03}  (see also \cite[Theorem 2.16]{Mag06}). Moreover if $x\in S$ is a non-characteristic point, there exists $U_x$ open subset of $S$ containing $x$ such that $U_x$ is a smooth non-characteristic hypersurface. Then we can use the previous argument to conclude that $(U_x,d)$ is $(\{\mathbb H^{n-1}\times \mathbb R\},\mathcal{H}^{2n+1})$-rectifiable and by covering $S\setminus \Sigma_S$ with countably many $U_x$'s we get the conclusion.
\end{proof}
\begin{remark}\label{rem:BiLipPauls}
	By \autoref{thm:FinalTheorem} it follows that any smooth hypersurface $S$ is $\mathbb H^{n-1}\times \mathbb R$-rectifiable according to the bi-Lipschitz variant of Pauls' definition \cite[Definition 3]{CP04}, see \autoref{rem:TangRectIsCarnotRect} for more details about this definition.
\end{remark}

\end{document}